\documentclass[12pt,reqno]{amsart}
\usepackage{amsmath,amsthm,amssymb,amsfonts,amscd}
\usepackage{mathrsfs}
\usepackage{bbding}
\usepackage{graphicx,latexsym}
\usepackage{hyperref}

\newcommand{\bea}{\begin{eqnarray}}
\newcommand{\eea}{\end{eqnarray}}
\newcommand{\bna}{\begin{eqnarray*}}
\newcommand{\ena}{\end{eqnarray*}}

\numberwithin{equation}{section} 

\theoremstyle{plain}
\newtheorem{theorem}{Theorem}[section]
\newtheorem{lemma}{Lemma}[section]

\newtheorem{proposition}{Proposition}[section]

\theoremstyle{definition}

\newtheorem{remark}{Remark}

\begin{document}

\title{Sums of the triple divisor function over values of a ternary quadratic form}

\author{Qingfeng Sun and Deyu Zhang}

\date{\today}

\begin{abstract}
Let $\tau_3(n)$ be the triple divisor function which
is the number of solutions of the equation $d_1d_2d_3=n$ in natural numbers.
It is shown that
$$
\sum_{1\leq n_1,n_2,n_3\leq \sqrt{x}}\tau_3(n_1^2+n_2^2+n_3^2)=c_1x^{\frac{3}{2}}(\log x)^2+
c_2x^{\frac{3}{2}}\log x
+c_3x^{\frac{3}{2}}
+O_{\varepsilon}(x^{\frac{11}{8}+\varepsilon})
$$
for some constants $c_1$, $c_2$ and $c_3$.
\end{abstract}

\maketitle
\tableofcontents

\section{Introduction}

The divisor functions
\bna
\tau_{k}(n)=\sum_{d_1\cdots d_k=n \atop d_1,\ldots,d_k\in \mathbb{Z^+}}1,
\ena
are the basic arithmetic functions in number theory, with the generating
Dirichlet series $\zeta^k(s)$ which are the simplest $GL_k$ $L$-functions.
While the Riemann zeta function $\zeta(s)$ has always been the most important
and intensively studied $L$-function, the behavior of $\tau_k(n)$ is far
less than perfectly understood even for $k=2$. For example, Hooley \cite{Hooley}
proved that
\bea
\sum_{n\leq x}\tau(n^2+a)=c_1x\log x+c_2x+O\left(x^{\frac{8}{9}}(\log x)^3\right)
\eea
for any fixed $a\in \mathbb{Z}$ such that $-a$ is not a perfect square, where
$c_1$ and $c_2$ are constants depending only on $a$. Here
as usual $\tau(n):=\tau_2(n)$. However, so far there are no
asymptotic formulas for the sum $\sum_{n\leq x}\tau(f(n))$ for $f(x)$ of degree $\mathrm{deg}f\geq 3$.
For the average behavior of the divisor functions over values of quadratic forms,
Yu \cite{Yu} proved that, as $x\rightarrow \infty$,
\bea
\sum_{1\leq n_1,n_2\leq \sqrt{x}}\tau(n_1^2+n_2^2)=
c_3x\log x+c_4x+O_{\varepsilon}\left(x^{\frac{3}{4}+\varepsilon}\right),
\eea
where $c_3$ and $c_4$ are constants.
Calder\'{o}n and de Velasco \cite{CV} studied the average behavior of $\tau(n)$ over values
of ternary quadratic form and established the asymptotic formula
\bea
\sum_{1\leq n_1,n_2,n_3\leq \sqrt{x} }\tau(n_1^2+n_2^2+n_3^2)
=\frac{4\zeta(3)}{5\zeta(5)}x^{\frac{3}{2}}\log x+O(x^{\frac{3}{2}}).
\eea
Recently, Guo and Zhai improved (1.3) by showing that
\bea
\sum_{1\leq n_1,n_2,n_3\leq \sqrt{x} }\tau(n_1^2+n_2^2+n_3^2)
=\frac{4\zeta(3)}{5\zeta(5)}x^{\frac{3}{2}}\log x+c_5x^{\frac{3}{2}}+O_{\varepsilon}(x^{\frac{4}{3}}),
\eea
where $c_5$ is a constant. The error term in (1.4) was further improved by Zhao \cite{Z} to
$O(x\log x)$. Nothing of type (1.1), (1.2) or (1.3) is known for $\tau_k(n)$ with $k\geq 3$ and
in fact the situation becomes even more difficult for $k\geq 3$ if one considers the sum
\bna
\sum_{n\leq x}a_n\tau_k(n)
\ena
for various sparse arithmetic sequences $a_n$. There are few results in this direction.
For $\tau_3(n)$, Friedlander and Iwaniec \cite{FI} showed that, for $x\geq 3$,
\bna
\sum_{n_1^2+n_2^6\leq x \atop (n_1,n_2)=1}\tau_3(n_1^2+n_2^6)=cx^{\frac{2}{3}}(\log x)^2
+O\left(x^{\frac{2}{3}}(\log x)^{\frac{7}{4}}(\log \log x)^{\frac{1}{2}}\right),
\ena
where $c$ is a constant.

In this paper, we want to prove an asymptotic formula of type (1.4) for $\tau_3(n)$.
Our main result is the following theorem.

\begin{theorem}
For any $x\geq x_0$ ($x_0$ is a large constant) and any $\varepsilon>0$, we have
\bna
\sum_{1\leq n_1,n_2,n_3\leq \sqrt{x}}\tau_3(n_1^2+n_2^2+n_3^2)
&=&\frac{\mathcal {C}_0\mathcal {J}_0}{4}x^{\frac{3}{2}}(\log x)^2+
\frac{1}{2}\left(\mathcal {C}_1\mathcal {J}_0+
\mathcal {C}_0\mathcal {J}_1\right)x^{\frac{3}{2}}\log x
\\
&&+
\frac{1}{2}\left(\mathcal {C}_2\mathcal {J}_0+
\mathcal {C}_1\mathcal {J}_1+\frac{1}{2}\mathcal {C}_0\mathcal {J}_2\right)x^{\frac{3}{2}}
+O_{\varepsilon}\left(x^{\frac{11}{8}+\varepsilon}\right),
\ena
where for $\ell=0,1,2$,
\bna
\mathcal {J}_{\ell}=\int_{-\infty}^{\infty}\left(\int_0^3(\log u)^{\ell}e(-\beta u)\mathrm{d}u\right)
\left(\int_0^1e(\beta v^2)\mathrm{d}v\right)^3\mathrm{d}\beta
\ena
and
\bna
\mathcal {C}_{\ell}=\sum_{q=1}^{\infty}\frac{1}{q^5}\sum_{n|q}n\tau(n)P_{\ell}(n,q)
\sum_{a=1\atop (a,q)=1}^qG(a,0;q)^3S(-\overline{a},0;q).
\ena
Here $\overline{a}$ denotes the multiplicative inverse of $a\bmod q$, $S(a,b;c)$ is the classical Kloosterman sum, $G(a,b;q)$ is the Gauss sum
\bna
G(a,b;q)=\sum\limits_{d\bmod q}e\left(\frac{ad^2+bd}{q}\right),
\ena
and $P_j(n,q)$ ($j=1,2$) are given by
\bna
P_1(n,q)&=&\frac{5}{3}\log n-3\log q+3\gamma-\frac{1}{3\tau(n)}\sum_{d|n}\log d,\\
P_2(n,q)&=&\left(\log n\right)^2-5\log q \log n
+\frac{9}{2}(\log q)^2+3\gamma^2-3\gamma_1+7\gamma \log n-9\gamma \log q\nonumber\\
&&+\frac{1}{\tau(n)}\left(\left(\log n+\log q-5\gamma\right)\sum_{d|n}\log d-\frac{3}{2}
\sum_{d|n}(\log d)^2\right)
\ena
with $\gamma:=\lim\limits_{s\rightarrow 1}\left(\zeta(s)-\frac{1}{s-1}\right)$
being the Euler constant and $\gamma_1:=-\frac{\mathrm{d}}{\mathrm{d}s}\left.\left(\zeta(s)-\frac{1}{s-1}\right)\right|_{s=1}$
being the Stieltjes constant.
\end{theorem}

A similar asymptotic formula can be derived for the slightly modified sum
\bna
\sum_{1\leq n_1^2+n_2^2+n_3^2\leq x \atop (n_1,n_2,n_3)\in \mathbb{Z}^3}\tau_3(n_1^2+n_2^2+n_3^2)
\ena
which is in some sense simpler than the sum in Theorem 1.1, since
obviously we have
\bea
\sum_{1\leq n_1^2+n_2^2+n_3^2\leq x}\tau_3(n_1^2+n_2^2+n_3^2)
=\sum_{1\leq n\leq x}\tau_3(n)r_3(n),
\eea
where
\bna
r_3(n)=\#\left\{(n_1,n_2,n_3)\in \mathbb{Z}^3:n_1^2+n_2^2+n_3^2=n\right\}.
\ena

\begin{theorem}
For any $x\geq x_0$ ($x_0$ is the same as that in Theorem 1.1) and any $\varepsilon>0$, we have
\bna
\sum_{1\leq n_1^2+n_2^2+n_3^2\leq x \atop (n_1,n_2,n_3)\in \mathbb{Z}^3}\tau_3(n_1^2+n_2^2+n_3^2)
&=&2\mathcal {C}_0\mathcal {K}_0x^{\frac{3}{2}}(\log x)^2+
4\left(\mathcal {C}_1\mathcal {K}_0+
\mathcal {C}_0\mathcal {K}_1\right)x^{\frac{3}{2}}\log x
\\
&&+
4\left(\mathcal {C}_2\mathcal {K}_0+
\mathcal {C}_1\mathcal {K}_1+\frac{1}{2}\mathcal {C}_0\mathcal {K}_2\right)x^{\frac{3}{2}}
+O_{\varepsilon}\left(x^{\frac{11}{8}+\varepsilon}\right),
\ena
where $C_{\ell}$'s are as in Theorem 1.1 and
\bna
\mathcal {K}_{\ell}=\int_{-\infty}^{\infty}\left(\int_0^1(\log u)^{\ell}e(-\beta u)\mathrm{d}u\right)
\left(\int_0^1e(\beta v^2)\mathrm{d}v\right)^3\mathrm{d}\beta.
\ena
\end{theorem}

\begin{remark}
In Theorems 1.1 and 1.2,
$x_0$ is an absolute constant which can be explicitly computed. We can take $x_0=4^8$.
\end{remark}

\begin{remark}
The proofs of Theorems 1.1 and 1.2 are partly similar as that in Sun \cite{Sun} and Zhao \cite{Z}.
The main saving comes from square-root cancelation of a two dimensional twisted
character sum (see Section 9) which benefits from the beautiful theorem of Fu \cite{Fu}.
\end{remark}

\begin{remark}
The proof of Theorem 1.2 is similar as that of Theorem 1.1 and we shall omit it for simplicity.
In view of (1.5), the error term in Theorem 1.2 may be further improved by appealing to the analytic
properties of the $L$-function
\bna
\sum_{n\geq 1}\frac{\tau_3(n)r_3(n)}{n^s}.
\ena
However, we will not take up this issue in this paper.
\end{remark}

\medskip

\noindent{\bf Notation.}
Throughout the paper, the letters $q$, $m$ and $n$, with or without subscript,
denote integers. The letter $\varepsilon$ is an arbitrarily small
positive constant, not necessarily the same at different occurrences. The symbol
$\ll_{a,b,c}$ denotes that the implied constant depends at most on $a$, $b$ and $c$.

\medskip

\section{Derivation of Theorem 1.1}
\setcounter{equation}{0}
\medskip

Let $\mathcal{V}$ denote the set $[1,\sqrt{x}]\cap \mathbb{Z}$ and $r_3^*(n)=\sum\limits_{n_1^2+n_2^2+n_3^2=n\atop
(n_1,n_2,n_3)\in \mathcal{V}^3}1$. By dyadic subdivision, we decompose the aimed sum
into partial sums
\bea
\sum_{1\leq n_1,n_2,n_3\leq \sqrt{x}}\tau_3(n_1^2+n_2^2+n_3^2)=
\sum_{j\geq 1}\sum_{3x/2^j<n\leq 3x/2^{j-1}}\tau_3(n)r_3^*(n).
\eea
Notice that the inner sum in (2.1) vanishes for $j>\log 6x/\log 2$. So we are treating
$O(\log x)$ sums of the form
\bna
\sum_{X_j/2<n\leq X_j}\tau_3(n)r_3^*(n),
\ena
where $X_j=3x/2^{j-1}$.
Let $\phi(y)$ be a smooth function supported on $[1/2,1]$,
identically equal 1 on $[1/2+M^{-1},1-M^{-1}]$
with $M>4$, and satisfy $\phi^{(j)}(y)\ll_j M^j$ for any integer $j\geq 0$. Then we have
\bea
\sum_{X_j/2<n\leq X_j}\tau_3(n)r_3^*(n)=
\sum_{n\geq 1}\tau_3(n)r_3^*(n)\phi\left(\frac{n}{X_j}\right)+O_{\varepsilon}(X_j^{\frac{3}{2}+\varepsilon}M^{-1}).
\eea
Here the $O$-term comes from the bounds $\tau_3(n)\ll_{\varepsilon}n^{\varepsilon}$ and
\bna
r_3^*(n)\leq r_3(n)=\sum_{n_1^2+n_2^2+n_3^2=n\atop (n_1,n_2,n_3)\in \mathbb{Z}^3}1
\ll_{\varepsilon} n^{\frac{1}{2}+\varepsilon}.
\ena
Thus it remains to study the smoothed sum
\bna
\mathscr{S}(X)=\sum_{n\geq 1}\tau_3(n)r_3^*(n)\phi\left(\frac{n}{X}\right).
\ena

We are going to prove the following asymptotic formula for $\mathscr{S}(X)$.

\begin{theorem}
For any $1<X\leq 3x$ and any $\varepsilon>0$, we have
\bna
\mathscr{S}(X)
&=&
\frac{1}{2}\mathcal{I}_0(X)\mathcal {C}_2x^{\frac{3}{2}}+\frac{1}{2}\mathcal{I}_1(X)\mathcal {C}_1x^{\frac{3}{2}}
+\frac{1}{4}\mathcal{I}_2(X)\mathcal {C}_0x^{\frac{3}{2}}+O_{\varepsilon}\left(x^{\frac{5}{4}+\varepsilon}M
+x^{\frac{3}{2}+\varepsilon}M^{-1}\right),
\ena
where
\bea
\mathcal {C}_{\ell}=\sum_{q=1}^{\infty}\frac{1}{q^5}\sum_{n|q}n\tau(n)P_{\ell}(n,q)
\sideset{}{^*}\sum_{a=1}^q
G(a,0;q)^3 S\left(-\overline{a},0;\frac{q}{n}\right),
\eea
and
\bna
\mathcal{I}_{\ell}(X)=\int_{-\infty}^{\infty}
\left(\int_{X/2}^Xe(-\beta u)(\log u)^{\ell}\mathrm{d}u\right)
\left(\int_0^{1}e(\beta x  v^2)\mathrm{d}v\right)^3\mathrm{d}\beta.
\ena
\end{theorem}

We set the proof of Theorem 2.1 aside and continue the derivation of Theorem 1.1.
Applying Theorem 2.1 to the sum on the right side of (2.2) and taking $M=x^{\frac{1}{8}}$ we get
\bea
\sum_{X_j/2<n\leq X_j}\tau_3(n)r_3^*(n)=
\frac{1}{2}\mathcal{I}_0(X_j)\mathcal {C}_2x^{\frac{3}{2}}+\frac{1}{2}\mathcal{I}_1(X_j)\mathcal {C}_1x^{\frac{3}{2}}
+\frac{1}{4}\mathcal{I}_2(X_j)\mathcal {C}_0x^{\frac{3}{2}}
+O_{\varepsilon}\left(x^{\frac{11}{8}+\varepsilon}
\right).
\eea
Set $j_0:=j_0(x)=[\log x/\log 2]$. Plugging (2.4) into (2.1) we obtain
\bea
&&\sum_{1\leq n_1,n_2,n_3\leq \sqrt{x}}\tau_3(n_1^2+n_2^2+n_3^2)\nonumber\\
&=&
\sum_{1\leq j\leq j_0}\sum_{3x/2^j<n\leq 3x/2^{j-1}}\tau_3(n)r_3^*(n)+O(1)\nonumber\\
&=&\frac{1}{2}\mathcal {C}_2x^{\frac{3}{2}}\sum_{1\leq j\leq j_0}\mathcal{I}_0\left(\frac{3x}{2^{j-1}}\right)
+\frac{1}{2}\mathcal {C}_1x^{\frac{3}{2}}\sum_{1\leq j\leq j_0}\mathcal{I}_1\left(\frac{3x}{2^{j-1}}\right)
+\frac{1}{4}\mathcal {C}_0x^{\frac{3}{2}}\sum_{1\leq j\leq j_0}\mathcal{I}_2\left(\frac{3x}{2^{j-1}}\right)
\nonumber\\
&&+O_{\varepsilon}\left(x^{\frac{11}{8}+\varepsilon}\right),\nonumber\\
\eea
where for $\ell=0,1,2$,
\bea
\sum_{1\leq j\leq j_0}\mathcal{I}_{\ell}\left(\frac{3x}{2^{j-1}}\right)&=&
\int_{-\infty}^{\infty}\left(\sum_{1\leq j\leq j_0}\int_{\frac{3x}{2^{j}}}^{\frac{3x}{2^{j-1}}}
(\log u)^{\ell}e(-\beta u)\mathrm{d}u\right)
\left(\int_0^1e(\beta x v^2)\mathrm{d}v\right)^3\mathrm{d}\beta\nonumber\\
&=&x\int_{-\infty}^{\infty}\left(\int_{\frac{3}{2^{j_0}}}^3(\log u x)^{\ell}
e(-\beta x u)\mathrm{d}u\right)
\left(\int_0^1e(\beta x v^2)\mathrm{d}v\right)^3\mathrm{d}\beta\nonumber\\
&=&\sum_{i=0}^{\ell}C_{\ell}^i(\log x)^{\ell-i}\int_{-\infty}^{\infty}\left(\int_0^3
(\log u)^ie(-\beta u)\mathrm{d}u\right)
\left(\int_0^1e(\beta v^2)\mathrm{d}v\right)^3\mathrm{d}\beta\nonumber\\
&&+\mathcal{R}_{\ell}(x),
\eea
and for $x\geq 6$,
\bea
\mathcal{R}_{\ell}(x)&=&\sum_{i=0}^{\ell}
C_{\ell}^i(\log x)^{\ell-i}\int_{-\infty}^{\infty}\left(\int_0^{\frac{3}{2^{j_0}}}
(\log u)^ie(-\beta u)\mathrm{d}u\right)
\left(\int_0^1e(\beta v^2)\mathrm{d}\beta\right)^3\mathrm{d}\beta\nonumber\\
&\ll_{\ell}&(\log x)^{\ell}\int_0^{\frac{3}{2^{j_0}}}
(-\log u)^{\ell}\mathrm{d}u\nonumber\\
&\ll_{\ell}&x^{-1}(\log x)^4.
\eea
By (2.5)-(2.7) we obtain
\bna
\sum_{1\leq n_1,n_2,n_3\leq \sqrt{x}}\tau_3(n_1^2+n_2^2+n_3^2)
&=&\frac{\mathcal {C}_0\mathcal {J}_0}{4}x^{\frac{3}{2}}(\log x)^2+
\frac{1}{2}\left(\mathcal {C}_1\mathcal {J}_0+
\mathcal {C}_0\mathcal {J}_1\right)x^{\frac{3}{2}}\log x\nonumber\\
&&+
\frac{1}{2}\left(\mathcal {C}_2\mathcal {J}_0+
\mathcal {C}_1\mathcal {J}_1+\frac{1}{2}\mathcal {C}_0\mathcal {J}_2\right)x^{\frac{3}{2}}
+O_{\varepsilon}\left(x^{\frac{11}{8}+\varepsilon}\right),
\ena
where for $\mathcal {C}_{\ell}$ is defined in (2.3) and
\bna
\mathcal {J}_{\ell}=\int_{-\infty}^{\infty}\left(\int_0^3(\log u)^{\ell}e(-\beta u)\mathrm{d}u\right)
\left(\int_0^1e(\beta v^2)\mathrm{d}v\right)^3\mathrm{d}\beta.
\ena
This finishes the proof of Theorem 1.1. The following sections are devoted to the proof of
Theorem 2.1.

\medskip

\section{Voronoi formula for the triple divisor function}
\setcounter{equation}{0}
\medskip

The Voronoi formula for $\tau_3(n)$ was first proved by Ivi\'{c} \cite{Iv} and
later in \cite{Li3}, Li derived a more explicit formula. To adopt Li's result, set
\bea
\sigma_{0,0}(k,l)=\sum_{d_1|l \atop d_1>0}\mathop{\sum_{d_2|\frac{l}{d_1} \atop d_2>0}}_{(d_2,k)=1}1.
\eea
Let $\zeta(s)$ be the Riemann zeta function, $\gamma:=\lim\limits_{s\rightarrow 1}\left(\zeta(s)-\frac{1}{s-1}\right)$
be the Euler constant and $\gamma_1:=-\frac{\mathrm{d}}{\mathrm{d}s}\left.\left(\zeta(s)-\frac{1}{s-1}\right)\right|_{s=1}$
be the Stieltjes constant. For $\phi(y)\in C_c(0,\infty)$, $k=0,1$ and $\sigma>-1-2k$, set
\bna
\Phi_k(y)=\frac{1}{2\pi i}\int\limits_{\mathrm{Re}(s)=\sigma}\left(\pi^3y\right)^{-s}
\frac{\Gamma\left(\frac{1+s+2k}{2}\right)^3}{\Gamma\left(\frac{-s}{2}\right)^3}
\widetilde{\phi}(-s-k)\mathrm{d}s
\ena
with $\widetilde{\phi}(s)=\int_0^{\infty}\phi(u)u^{s-1}\mathrm{d}u$ the Mellin transform of $\phi$, and
\bna
\Phi^{\pm}(y)=\Phi_0(y)\pm \frac{1}{i\pi^3y}\Phi_1(y).
\ena

\medskip

\begin{lemma} For $\phi(y)\in C_c^{\infty}(0,\infty)$, $a, \overline{a}, q \in \mathbb{Z}^+$ with
$a\overline{a}\equiv 1(\bmod q)$, we have
\bna
&&\sum_{n\geq 1}\tau_3(n)e\left(\frac{an}{q}\right)\phi(n)\\
&=&
\frac{q}{2\pi^{\frac{3}{2}}}\sum_{\pm}\sum_{n|q}\sum_{m\geq 1}\frac{1}{nm}
\sum_{n_1|n}\sum_{n_2|\frac{n}{n_1}}\sigma_{0,0}\left(\frac{n}{n_1n_2},m\right)
S\left(\pm m,\overline{a};\frac{q}{n}\right)\Phi^{\pm}\left(\frac{mn^2}{q^3}\right)\\
&&+\frac{1}{2q^2}\widetilde{\phi}(1)\sum_{n|q}n\tau(n)P_2(n,q)S\left(0,\overline{a};\frac{q}{n}\right)
\\
&&+\frac{1}{2q^2}\widetilde{\phi}'(1)\sum_{n|q}n\tau(n)P_1(n,q)S\left(0,\overline{a};\frac{q}{n}\right)
\\
&&+\frac{1}{4q^2}\widetilde{\phi}''(1)\sum_{n|q}n\tau(n)S\left(0,\overline{a};\frac{q}{n}\right),
\ena
where $S(a,b;c)$ is the classical Kloosterman sum,
\bea
P_1(n,q)=\frac{5}{3}\log n-3\log q+3\gamma-\frac{1}{3\tau(n)}\sum_{d|n}\log d,
\eea
and
\bea
P_2(n,q)&=&\left(\log n\right)^2-5\log q \log n
+\frac{9}{2}(\log q)^2+3\gamma^2-3\gamma_1+7\gamma \log n-9\gamma \log q\nonumber\\
&&+\frac{1}{\tau(n)}\left(\left(\log n+\log q-5\gamma\right)\sum_{d|n}\log d-\frac{3}{2}
\sum_{d|n}(\log d)^2\right).
\eea
\end{lemma}

The functions $\Phi^{\pm}(y)$ have the following properties (see Sun \cite{Sun}).

\begin{lemma}
Suppose that $\phi(y)$ is a smooth function of
compact support in $[AX,BX]$,
where $X>0$ and $B>A>0$,
satisfying $\phi^{(j)}(y)\ll_{A,B,j}P^j $ for any integer $j\geq 0$.
Then for $y>0$ and any integer $\ell\geq 0$, we have
$$
\Phi^{\pm}(y)\ll_{A,B,\ell,\varepsilon}(yX)^{-\varepsilon}(PX)^3\left(\frac{y}{P^3X^2}\right)^{-\ell}.
$$
\end{lemma}

By Lemma 3.2, for any fixed $\varepsilon>0$ and $yX\geq X^{\varepsilon}(P X)^3$,
$\Phi^{\pm}(y)$ are negligibly small.
Moreover, for $yX\gg X^{\varepsilon}$, we have an asymptotic formula for
$\Phi_k(y)$ (see \cite{Iv}, \cite{Li1}, \cite{RY}).

\begin{lemma}
Suppose that $\phi(y)$ is a smooth function of compact support on $[AX,BX]$,
where $X>0$ and $B>A>0$. Then for $y>0$, $yX\gg 1$, $\ell\geq 2$ and $k=0,1$, we have
\bna
\Phi_k(y)&=&(\pi^3 y)^{k+1}\sum_{j=1}^{\ell} \int_0^{\infty}\phi(u)
\left(a_k(j)e\left(3(y u)^{\frac{1}{3}}\right)+b_k(j)e\left(-3(y u)^{\frac{1}{3}}\right)\right)
\frac{\mathrm{d}u}{(\pi^3yu)^{\frac{j}{3}}}\\
&&+O_{A,B,\varepsilon,\ell}\left((\pi^3y)^k(\pi^3yX)^{-\frac{\ell}{3}+\frac{1}{2}+\varepsilon}\right),
\ena
where $a_k(j)$, $b_k(j)$ are constants with
\bna
a_0(1)=-\frac{2\sqrt{3\pi}}{6\pi i}, \quad b_0(1)=\frac{2\sqrt{3\pi}}{6\pi i},
\quad a_1(1)=b_1(1)=-\frac{2\sqrt{3\pi}}{6\pi}.
\ena
\end{lemma}

\medskip

\section{Transformation of $\mathscr{S}(X)$}
\setcounter{equation}{0}
\medskip

Applying the circle method, we have
\bna
\mathscr{S}(X)=\int_0^1\mathscr{F}^3(\alpha)\mathscr{G}(\alpha)\mathrm{d}\alpha,
\ena
where
\bna
\mathscr{F}(\alpha)=\sum_{n\in \mathcal{V}}e(\alpha n^2),
\ena
and
\bea
\mathscr{G}(\alpha)=\sum_{n\geq 1}\tau_3(n)e(-\alpha n)\phi\left(\frac{n}{X}\right).
\eea
Note that $\mathscr{F}^3(\alpha)\mathscr{G}(\alpha)$ is a periodic function of period 1. We have
\bna
\mathscr{S}(X)=\int_{-1/(Q+1)}^{Q/(Q+1)}\mathscr{F}^3(\alpha)\mathscr{G}(\alpha)\mathrm{d}\alpha,
\ena
where $Q$ is a large integer to be chosen later.
Then we can evaluate $\mathscr{S}(X)$ by
dissecting the interval $\left(-1/(Q+1),Q/(Q+1)\right]$ with Farey's points
of order $Q$ (see for example Iwaniec \cite{Iw}). Let
$\frac{a'}{q'}<\frac{a}{q}<\frac{a''}{q''}$ be adjacent points, which are determined
by the conditions
\[
Q<q+q',q+q''\leq q+Q, \quad aq'\equiv 1(\bmod q), \quad aq''\equiv -1(\bmod q).
\]
Then
\bna
\left(\frac{-1}{Q+1},\frac{Q}{Q+1}\right]=\mathop{\bigcup\bigcup}_{0\leq a<q\leq Q\atop (a,q)=1}
\left(\frac{a}{q}-\frac{1}{q(q+q')},\frac{a}{q}+\frac{1}{q(q+q'')}\right].
\ena
It follows that
\bna
\mathscr{S}(X)=\sum_{q\leq Q}\sideset{}{^*}\sum_{a=1}^q
\int\limits_{\mathscr{M}(a,q)}\mathscr{F}^3\left(\frac{a}{q}+\beta\right)\mathscr{G}\left(\frac{a}{q}+\beta\right)\mathrm{d}\beta,
\ena
where $*$ denotes the condition $(a,q)=1$ and
\[
\mathscr{M}(a,q)=\left(-\frac{1}{q(q+q')},\frac{1}{q(q+q'')}\right].
\]

Exchanging the order of the summation over $a$ and the integration over $\beta$ as in
Heath-Brown \cite{HB}, we have
\bea
\mathscr{S}(X)&=&\sum_{q\leq Q}\int\limits_{|\beta|\leq \frac{1}{qQ}}\sum_{v\bmod q}
\varrho(v,q,\beta)\sideset{}{^*}\sum_{a=1}^qe\left(-\frac{\overline{a}v}{q}\right)
\mathscr{F}^3\left(\frac{a}{q}+\beta\right)\mathscr{G}\left(\frac{a}{q}+\beta\right)\mathrm{d}\beta,
\eea
where $\varrho(v,q,\beta)$ satisfies
\bea
\varrho(v,q,\beta)\ll \frac{1}{1+|v|}.
\eea

For an asymptotic formula of $\mathscr{F}\left(\frac{a}{q}+\beta\right)$,
we quote the following result (see Theorem 4.1 in \cite{V} or Lemma 4.1 in \cite{Z}).

\begin{lemma}
Let $Q=[5\sqrt{x}]$. Suppose that $(a,q)=1$, $q\leq Q$ and
$|\beta|\leq 1/(qQ)$. We have
\bea
\mathscr{F}\left(\frac{a}{q}+\beta\right)=\frac{G(a,0;q)}{q}\Psi_0(\beta)
+\sum_{-\frac{3q}{2}<b\leq \frac{3q}{2}}G(a,b;q)\Psi(b,q,\beta),
\eea
where $G(a,b;q)$ is the Gauss sum
\bea
G(a,b;q)=\sum\limits_{d \bmod q}e\left(\frac{ad^2+bd}{q}\right),
\eea
$\Psi_0(\beta)$ is the integral
\bea
\Psi_0(\beta)=\int_0^{\sqrt{x}}e(\beta u^2)\mathrm{d}u,
\eea
and $\Psi(b,q,\beta)$ satisfies
\bea
\sum_{-\frac{3q}{2}<b\leq \frac{3q}{2}}|\Psi(b,q,\beta)|\ll \log(q+2).
\eea
\end{lemma}

For $\mathscr{G}(\alpha)$ in (4.1), we apply Lemma 3.1
with $\phi_{\beta}(y)=\phi\left(\frac{y}{X}\right)e(-\beta y)$ getting
\bea
\mathscr{G}\left(\frac{a}{q}+\beta\right)
&=&\sum_{n\geq 1}\tau_3(n)e\left(-\frac{a n}{q}\right)
\phi_{\beta}(n)\nonumber\\
&=&\frac{q}{2\pi^{\frac{3}{2}}}\sum_{\pm}\sum_{n|q}
\sum_{m\geq 1}\frac{1}{nm}
\sum_{n_1|n}\sum_{n_2|\frac{n}{n_1}}\sigma_{0,0}\left(\frac{n}{n_1n_2},m\right)
S\left(\pm m,-\overline{a};\frac{q}{n}\right)
\Phi_{\beta}^{\pm}\left(\frac{mn^2}{q^3}\right)\nonumber\\
&+&\frac{1}{2q^2}\widetilde{\phi}_{\beta}(1)\sum_{n|q}
n\tau(n)P_2(n,q)
S\left(0,-\overline{a};\frac{q}{n}\right)\nonumber\\
&+&\frac{1}{2q^2}\widetilde{\phi}'_{\beta}(1)\sum_{n|q}
n\tau(n)P_1(n,q)
S\left(0,-\overline{a};\frac{q}{n}\right)
\nonumber\\
&+&\frac{1}{4q^2}\widetilde{\phi}_{\beta}''(1)\sum_{n|q}n\tau(n)
S\left(0,-\overline{a};\frac{q}{n}\right),
\eea
where
\bea
\Phi_{\beta}^{\pm}(y)=\Phi_0(y,\beta)\pm\frac{1}{i\pi^3y}\Phi_1(y,\beta)
\eea
with
\bea
\Phi_k(y,\beta)=\frac{1}{2\pi i}\int\limits_{\mathrm{Re}(s)=\sigma}\left(\pi^3y\right)^{-s}
\frac{\Gamma\left(\frac{1+s+2k}{2}\right)^3}{\Gamma\left(\frac{-s}{2}\right)^3}
\widetilde{\phi}_{\beta}(-s-k)\mathrm{d}s,
\eea
and $P_j(n,q)$ ($j=1,2$) defined in (3.2) and (3.3).

By (4.4) and (4.8), we have
\bea
\sideset{}{^*}\sum_{a=1}^qe\left(-\frac{\overline{a}v}{q}\right)
\mathscr{F}^3\left(\frac{a}{q}+\beta\right)\mathscr{G}\left(\frac{a}{q}+\beta\right)
=\sum_{j=1}^{16}\mathscr{B}_j(v,q,\beta),
\eea
where
\bea
\mathscr{B}_1(v,q,\beta)&=&\frac{q}{2\pi^{\frac{3}{2}}}\sum_{\pm}
\sum_{n|q}\sum_{m\geq 1}\frac{1}{nm}\sum_{n_1|n}
\sum_{n_2|\frac{n}{n_1}}\sigma_{0,0}\left(\frac{n}{n_1n_2},m\right)
\Phi_{\beta}^{\pm}\left(\frac{mn^2}{q^3}\right)\nonumber\\
&&\times
\sum_{-\frac{3q}{2}<b_j\leq \frac{3q}{2}\atop 1\leq j\leq 3}
\Psi(b_1,q,\beta)\Psi(b_2,q,\beta)\Psi(b_3,q,\beta)
\mathscr{C}(b_1,b_2,b_3,n,m,v;q),\\
\mathscr{B}_2(v,q,\beta)&=&\frac{3}{2\pi^{\frac{3}{2}}}\Psi_0(\beta)
\sum_{\pm}
\sum_{n|q}\sum_{m\geq 1}\frac{1}{nm}\sum_{n_1|n}
\sum_{n_2|\frac{n}{n_1}}\sigma_{0,0}\left(\frac{n}{n_1n_2},m\right)
\Phi_{\beta}^{\pm}\left(\frac{mn^2}{q^3}\right)\nonumber\\
&&\times\sum_{-\frac{3q}{2}<b_j\leq \frac{3q}{2}\atop j=1,2}\Psi(b_1,q,\beta)\Psi(b_2,q,\beta)
\mathscr{C}(0,b_1,b_2,n,m,v;q),\\
\mathscr{B}_3(v,q,\beta)&=&\frac{3}{2\pi^{\frac{3}{2}}}\frac{\Psi_0(\beta)^2}{q}
\sum_{\pm}
\sum_{n|q}\sum_{m\geq 1}\frac{1}{nm}\sum_{n_1|n}\sum_{n_2|\frac{n}{n_1}}
\sigma_{0,0}\left(\frac{n}{n_1n_2},m\right)
\Phi_{\beta}^{\pm}\left(\frac{mn^2}{q^3}\right)\nonumber\\
&&\times\sum_{-\frac{3q}{2}<b\leq \frac{3q}{2}}\Psi(b,q,\beta)
\mathscr{C}(0,0,b,n,m,v;q),\\
\mathscr{B}_4(v,q,\beta)&=&\frac{1}{2\pi^{\frac{3}{2}}}\frac{\Psi_0(\beta)^3}{q^2}
\sum_{\pm}
\sum_{n|q}\sum_{m\geq 1}\frac{1}{nm}\sum_{n_1|n}\sum_{n_2|\frac{n}{n_1}}
\sigma_{0,0}\left(\frac{n}{n_1n_2},m\right)\Phi_{\beta}^{\pm}\left(\frac{mn^2}{q^3}\right)\nonumber\\\nonumber\\
&&\times
\mathscr{C}(0,0,0,n,m,v;q),\\
\mathscr{B}_5(v,q,\beta)&=&\frac{1}{2}\frac{\widetilde{\phi}_{\beta}(1)}{q^2}
\sum_{n|q}n\tau(n)P_2(n,q)
\sum_{-\frac{3q}{2}<b_j\leq \frac{3q}{2}\atop 1\leq j\leq 3}
\Psi(b_1,q,\beta)\Psi(b_2,q,\beta)\Psi(b_3,q,\beta)\nonumber\\
&&\times
\mathscr{C}(b_1,b_2,b_3,n,0,v;q),\\
\mathscr{B}_6(v,q,\beta)&=&\frac{1}{2}\frac{\widetilde{\phi}'_{\beta}(1)}{q^2}
\sum_{n|q}n\tau(n)P_1(n,q)
\sum_{-\frac{3q}{2}<b_j\leq \frac{3q}{2}\atop 1\leq j\leq 3}
\Psi(b_1,q,\beta)\Psi(b_2,q,\beta)\Psi(b_3,q,\beta)\nonumber\\
&&\times
\mathscr{C}(b_1,b_2,b_3,n,0,v;q),\\
\mathscr{B}_7(v,q,\beta)&=&\frac{1}{4}\frac{\widetilde{\phi}''_{\beta}(1)}{q^2}
\sum_{n|q}n\tau(n)
\sum_{-\frac{3q}{2}<b_j\leq \frac{3q}{2}\atop 1\leq j\leq 3}
\Psi(b_1,q,\beta)\Psi(b_2,q,\beta)\Psi(b_3,q,\beta)\nonumber\\
&&\times
\mathscr{C}(b_1,b_2,b_3,n,0,v;q),\\
\mathscr{B}_8(v,q,\beta)&=&\frac{3}{2}\frac{\widetilde{\phi}_{\beta}(1)\Psi_0(\beta)}{q^3}
\sum_{n|q}n\tau(n)P_2(n,q)
\sum_{-\frac{3q}{2}<b_j\leq \frac{3q}{2}\atop j=1,2}
\Psi(b_1,q,\beta)\Psi(b_2,q,\beta)\nonumber\\
&&\times
\mathscr{C}(0,b_1,b_2,n,0,v;q),
\eea
\bea
\mathscr{B}_9(v,q,\beta)&=&\frac{3}{2}\frac{\widetilde{\phi}'_{\beta}(1)\Psi_0(\beta)}{q^3}
\sum_{n|q}n\tau(n)P_1(n,q)
\sum_{-\frac{3q}{2}<b_j\leq \frac{3q}{2}\atop j=1,2}
\Psi(b_1,q,\beta)\Psi(b_2,q,\beta)\nonumber\\
&&\times
\mathscr{C}(0,b_1,b_2,n,0,v;q),\\
\mathscr{B}_{10}(v,q,\beta)&=&\frac{3}{4}\frac{\widetilde{\phi}''_{\beta}(1)\Psi_0(\beta)}{q^3}
\sum_{n|q}n\tau(n)
\sum_{-\frac{3q}{2}<b_j\leq \frac{3q}{2}\atop j=1,2}
\Psi(b_1,q,\beta)\Psi(b_2,q,\beta)\nonumber\\
&&\times
\mathscr{C}(0,b_1,b_2,n,0,v;q),\\
\mathscr{B}_{11}(v,q,\beta)&=&\frac{3}{2}\frac{\widetilde{\phi}_{\beta}(1)\Psi_0(\beta)^2}{q^4}
\sum_{n|q}n\tau(n)P_2(n,q)
\sum_{-\frac{3q}{2}<b\leq \frac{3q}{2}}
\Psi(b,q,\beta)\nonumber\\
&&\times
\mathscr{C}(0,0,b,n,0,v;q),\\
\mathscr{B}_{12}(v,q,\beta)&=&\frac{3}{2}\frac{\widetilde{\phi}'_{\beta}(1)\Psi_0(\beta)^2}{q^4}
\sum_{n|q}n\tau(n)P_1(n,q)
\sum_{-\frac{3q}{2}<b\leq \frac{3q}{2}}
\Psi(b,q,\beta)\nonumber\\
&&\times
\mathscr{C}(0,0,b,n,0,v;q),\\
\mathscr{B}_{13}(v,q,\beta)&=&\frac{3}{4}\frac{\widetilde{\phi}''_{\beta}(1)\Psi_0(\beta)^2}{q^4}
\sum_{n|q}n\tau(n)
\sum_{-\frac{3q}{2}<b\leq \frac{3q}{2}}
\Psi(b,q,\beta)\nonumber\\
&&\times
\mathscr{C}(0,0,b,n,0,v;q),\\
\mathscr{B}_{14}(v,q,\beta)&=&\frac{1}{2}\frac{\widetilde{\phi}_{\beta}(1)\Psi_0(\beta)^3}{q^5}
\sum_{n|q}n\tau(n)P_2(n,q)
\mathscr{C}(0,0,0,n,0,v;q),\\
\mathscr{B}_{15}(v,q,\beta)&=&\frac{1}{2}\frac{\widetilde{\phi}'_{\beta}(1)\Psi_0(\beta)^3}{q^5}
\sum_{n|q}n\tau(n)P_1(n,q)
\mathscr{C}(0,0,0,n,0,v;q),\\
\mathscr{B}_{16}(v,q,\beta)&=&\frac{1}{4}\frac{\widetilde{\phi}''_{\beta}(1)\Psi_0(\beta)^3}{q^5}
\sum_{n|q}n\tau(n)
\mathscr{C}(0,0,0,n,0,v;q)
\eea
with
\bea
\mathscr{C}(b_1,b_2,b_3,n,m,v;q)&=&\sideset{}{^*}\sum_{a=1}^q
e\left(\frac{-\overline{a}v}{q}\right)
G(a,b_1;q)G(a,b_2;q)G(a,b_3;q) S\left(-\overline{a},\pm m;\frac{q}{n}\right).
\nonumber\\
\eea

We will show that $\mathscr{B}_j$, $1\leq j\leq 13$, contribute the remainder terms, and
$\mathscr{B}_j$, $14\leq j\leq 16$, contribute the main terms.

\medskip

\section{Contribution of $\mathscr{B}_j$, $1\leq j\leq 4$}
\setcounter{equation}{0}
\medskip

The estimation of $\mathscr{B}_j$, $1\leq j\leq 4$, are similar
as the arguments in \cite{Sun}.
Since the cancelation from the character sums $\mathscr{C}(b_1,b_2,b_3,n,m,v;q)$
is the main saving for our final result, we first have the following proposition.

\begin{proposition}
Let $q=q_1q_0q_3'$, $q_1|n$,
$4q_0$ square-full and $q_3'$ square-free.
For any $\varepsilon>0$, we have
\bea
\mathscr{C}(b_1,b_2,b_3,n,m,v;q)\ll_{\varepsilon} \frac
{(q_1q_0)^{3+\varepsilon}q_3'^{\frac{5}{2}+\varepsilon}}{\sqrt{n}}.
\eea
\end{proposition}

Next, we need the following result which will be proved in Section 8.
\begin{proposition}
For any $\varepsilon>0$, we have
\bna
\sum_{\pm}\sum_{m\geq 1}\frac{1}{m}\sum_{n_1|n}\sum_{n_2|\frac{n}{n_1}}
\sigma_{0,0}\left(\frac{n}{n_1n_2},m\right)
\left|\Phi_{\beta}^{\pm}\left(\frac{mn^2}{q^3}\right)\right|\ll_{\varepsilon}
X^{\varepsilon}n^{\varepsilon}(M+|\beta|^2X^2).
\ena
\end{proposition}

By the second derivative test and the trivial estimation, $\Psi_0(\beta)$ in (4.6) is bounded by
\bea
\Psi_0(\beta)\ll \left(\frac{x}{1+|\beta|x}\right)^{\frac{1}{2}}.
\eea

Let $q$ be as in Proposition 5.1. By (4.7), (4.12) and Propositions 5.1-5.2, we have
\bna
\mathscr{B}_1(v,q,\beta)
&\ll& q\sum_{\pm}
\sum_{n|q}\sum_{m\geq 1}
\frac{1}{nm}\sum_{n_1|n}\sum_{n_2|\frac{n}{n_1}}
\sigma_{0,0}\left(\frac{n}{n_1n_2},m\right)
\left|\Phi_{\beta}^{\pm}\left(\frac{mn^2}{q^3}\right)\right|\nonumber\\
&&\times\sum_{-\frac{3q}{2}<b_j\leq \frac{3q}{2}\atop 1\leq j\leq 3}
|\Psi(b_1,q,\beta)||\Psi(b_2,q,\beta)|
|\Psi(b_3,q,\beta)|
|\mathscr{C}(b_1,b_2,b_3,n,m,v;q)|\nonumber\\
&\ll_{\varepsilon}&X^{\varepsilon}
\sum_{n\leq Q}n^{-\frac{3}{2}}
q_1^4q_0^4q_3'^{\frac{7}{2}}\sum_{\pm}\sum_{m\geq 1}\frac{1}{m}\sum_{n_1|n}\sum_{n_2|\frac{n}{n_1}}
\sigma_{0,0}\left(\frac{n}{n_1n_2},m\right)
\left|\Phi_{\beta}^{\pm}\left(\frac{mn^2}{q^3}\right)\right|\nonumber\\
&\ll_{\varepsilon}&X^{\varepsilon}(M+|\beta|^2X^2)
\sum_{n\leq Q}n^{-\frac{3}{2}+\varepsilon}
q_1^4q_0^4q_3'^{\frac{7}{2}}.
\ena
Then by (4.3), we have
\bea
&&\sum_{q\leq Q}\int\limits_{|\beta|\leq \frac{1}{q Q}}\sum_{v\bmod q}
\varrho(v,q,\beta)\mathscr{B}_1(v,q,\beta)\mathrm{d}\beta\nonumber\\
&\ll_{\varepsilon}&X^{\varepsilon}
\sum_{n\leq Q}n^{-\frac{3}{2}+\varepsilon}
\sum_{q_1|n}\sum_{q_3'\leq Q/q_1 \atop q_3' \,\mathrm{square-free}}
\sum_{q_0\leq Q/(q_1q_3') \atop 4q_0 \, \mathrm{square-full}}
q_1^4q_0^4q_3'^{\frac{7}{2}}
\int\limits_{|\beta|\leq \frac{1}{q_1q_0q_3' Q}}(M+|\beta|^2X^2)\mathrm{d}\beta\nonumber\\
&\ll_{\varepsilon}&X^{\varepsilon}\sum_{n\leq Q}
n^{-\frac{3}{2}+\varepsilon}
\sum_{q_1|n}\sum_{q_3'\leq Q/q_1}
\sum_{q_0\leq Q/(q_1q_3') \atop 4q_0 \, \mathrm{square-full}}
q_1^4q_0^4q_3'^{\frac{7}{2}}
 \left(\frac{M}{q_1q_0q_3'Q}+\frac{X^2}{(q_1q_0q_3'Q)^3} \right)\nonumber\\
&\ll_{\varepsilon}&\frac{X^{\varepsilon}M}{Q}
\sum_{n\leq Q}n^{-\frac{3}{2}+\varepsilon}
\sum_{q_1|n}q_1^3\sum_{q_3'\leq Q/q_1}q_3'^{\frac{5}{2}}
\left(\frac{Q}{q_1q_3'}\right)^{\frac{7}{2}}\nonumber\\
&&+\frac{X^{2+\varepsilon}}{Q^3}
\sum_{n\leq Q}n^{-\frac{3}{2}+\varepsilon}
\sum_{q_1|n}q_1\sum_{q_3'\leq Q/q_1}q_3'^{\frac{1}{2}}\left(\frac{Q}{q_1q_3'}\right)^{\frac{3}{2}}\nonumber\\
&\ll_{\varepsilon}&X^{\varepsilon}MQ^{\frac{5}{2}}
+\frac{X^{2+\varepsilon}}{Q^{\frac{3}{2}}}\nonumber\\
&\ll_{\varepsilon}&Mx^{\frac{5}{4}+\varepsilon}.
\eea

Moreover, by (4.7), (4.13) and Propositions 5.1-5.2, we have
\bna
\mathscr{B}_2(v,q,\beta)
&\ll_{\varepsilon}&x^{\varepsilon}\left(\frac{x}{1+|\beta|x}\right)^{\frac{1}{2}}
\sum_{n\leq Q}
n^{-\frac{3}{2}}q_1^3q_0^3q_3'^{\frac{5}{2}}\,X^{\varepsilon}n^{\varepsilon}(M+|\beta|^2X^2)\nonumber\\
&\ll_{\varepsilon}&x^{\frac{1}{2}+\varepsilon}\left(\left({1+|\beta|x}\right)^{\frac{3}{2}}
+\frac{M}{\sqrt{1+|\beta|x}}\right)
\sum_{n\leq Q}n^{-\frac{3}{2}+\varepsilon}
q_1^3q_0^3q_3'^{\frac{5}{2}}.
\ena
It follows that
\bea
&&\sum_{q\leq Q}\int\limits_{|\beta|\leq \frac{1}{q Q}}\sum_{v\bmod q}
\varrho(v,q,\beta)\mathscr{B}_2(v,q,\beta)\mathrm{d}\beta\nonumber\\
&\ll_{\varepsilon}&x^{\frac{1}{2}+\varepsilon}
\sum_{n\leq Q}
n^{-\frac{3}{2}+\varepsilon}
\sum_{q_1|n}\sum_{q_3'\leq Q/q_1 \atop q_3' \,\mathrm{square-free}}
\sum_{q_0\leq Q/(q_1q_3') \atop 4q_0 \, \mathrm{square-full}}
q_1^3q_0^3q_3'^{\frac{5}{2}}
\int\limits_{|\beta|\leq \frac{1}{q_1q_0q_3' Q}}(1+|\beta|x)^{\frac{3}{2}}\mathrm{d}\beta\nonumber\\
&&+x^{\varepsilon}M
\sum_{n\leq Q}
n^{-\frac{3}{2}+\varepsilon}
\sum_{q_1|n}\sum_{q_3'\leq Q/q_1 \atop q_3' \,\mathrm{square-free}}
\sum_{q_0\leq Q/(q_1q_3') \atop 4q_0 \, \mathrm{square-full}}
q_1^3q_0^3q_3'^{\frac{5}{2}}
\int\limits_{|\beta|\leq \frac{1}{q_1q_0q_3' Q}}\frac{1}{\sqrt{x^{-1}+|\beta|}}\mathrm{d}\beta
\nonumber
\eea
\bea
&\ll_{\varepsilon}&x^{\frac{1}{2}+\varepsilon}
\sum_{n\leq Q}
n^{-\frac{3}{2}+\varepsilon}
\sum_{q_1|n}\sum_{q_3'\leq Q/q_1 }
\sum_{q_0\leq Q/(q_1q_3') \atop 4q_0 \, \mathrm{square-full}}
q_1^3q_0^3q_3'^{\frac{5}{2}}
 \left(\frac{1}{q_1q_0q_3'Q}+\frac{x^{\frac{3}{2}}}{(q_1q_0q_3'Q)^{\frac{5}{2}}} \right)\nonumber\\
&&+x^{\varepsilon}M
\sum_{n\leq Q}
n^{-\frac{3}{2}+\varepsilon}
\sum_{q_1|n}\sum_{q_3'\leq Q/q_1 }
\sum_{q_0\leq Q/(q_1q_3') \atop 4q_0 \, \mathrm{square-full}}
q_1^3q_0^3q_3'^{\frac{5}{2}}(q_1q_0q_3'Q)^{-\frac{1}{2}}\nonumber\\
&\ll_{\varepsilon}&\frac{x^{\frac{1}{2}+\varepsilon}}{Q}
\sum_{n\leq Q}
n^{-\frac{3}{2}+\varepsilon}
\sum_{q_1|n}q_1^2\sum_{q_3'\leq Q/q_1}q_3'^{\frac{3}{2}}
\left(\frac{Q}{q_1q_3'}\right)^{\frac{5}{2}}\nonumber\\
&&+\frac{x^{2+\varepsilon}}{Q^{\frac{5}{2}}}
\sum_{n\leq Q}
n^{-\frac{3}{2}+\varepsilon}
\sum_{q_1|n}q_1^{\frac{1}{2}}\sum_{q_3'\leq Q/q_1}\frac{Q}{q_1q_3'}\nonumber\\
&&+\frac{x^{\varepsilon}M}{Q^{\frac{1}{2}}}
\sum_{n\leq Q}
n^{-\frac{3}{2}+\varepsilon}
\sum_{q_1|n}q_1^{\frac{5}{2}}\sum_{q_3'\leq Q/q_1}q_3'^2
\left(\frac{Q}{q_1q_3'}\right)^3\nonumber\\
&\ll_{\varepsilon}&x^{\frac{1}{2}+\varepsilon}Q^{\frac{3}{2}}
+\frac{x^{2+\varepsilon}}{Q^{\frac{3}{2}}}+x^{\varepsilon}MQ^{\frac{5}{2}}\nonumber\\
&\ll_{\varepsilon}&Mx^{\frac{5}{4}+\varepsilon}.
\eea

Further, by (4.7), (4.14) and Propositions 5.1-5.2, we have
\bna
\mathscr{B}_3(v,q,\beta)
&\ll_{\varepsilon}&
\frac{x^{1+\varepsilon}}{1+|\beta|x}
\sum_{n\leq Q}
n^{-\frac{3}{2}}
q_1^2q_0^2q_3'^{\frac{3}{2}}\,X^{\varepsilon}n^{\varepsilon}(M+|\beta|^2X^2)\nonumber\\
&\ll_{\varepsilon}&
x^{1+\varepsilon}\left(1+|\beta|x+\frac{M}{1+|\beta|x}\right)
\sum_{n\leq Q}
n^{-\frac{3}{2}+\varepsilon}
q_1^2q_0^2q_3'^{\frac{3}{2}}.
\ena
It follows from this estimate and (4.3) that
\bea
&&\sum_{q\leq Q}\int\limits_{|\beta|\leq \frac{1}{q Q}}\sum_{v\bmod q}
\varrho(v,q,\beta)\mathscr{B}_3(v,q,\beta)\mathrm{d}\beta\nonumber\\
&\ll_{\varepsilon}&x^{1+\varepsilon}\sum_{n\leq Q}
n^{-\frac{3}{2}+\varepsilon}
\sum_{q_1|n}\sum_{q_3'\leq Q/q_1 \atop q_3' \,\mathrm{square-free}}
\sum_{q_0\leq Q/(q_1q_3') \atop 4q_0 \, \mathrm{square-full}}
q_1^2q_0^2q_3'^{\frac{3}{2}}
\int\limits_{|\beta|\leq \frac{1}{q_1q_0q_3' Q}}(1+|\beta|x)\mathrm{d}\beta\nonumber\\
&&+x^{\varepsilon}M\sum_{n\leq Q}
n^{-\frac{3}{2}+\varepsilon}
\sum_{q_1|n}\sum_{q_3'\leq Q/q_1 \atop q_3' \,\mathrm{square-free}}
\sum_{q_0\leq Q/(q_1q_3') \atop 4q_0 \, \mathrm{square-full}}
q_1^2q_0^2q_3'^{\frac{3}{2}}
\int\limits_{|\beta|\leq \frac{1}{q_1q_0q_3' Q}}(x^{-1}+|\beta|)^{-1}\mathrm{d}\beta\nonumber
\eea
\bea
&\ll_{\varepsilon}&x^{1+\varepsilon}\sum_{n\leq Q}
n^{-\frac{3}{2}+\varepsilon}
\sum_{q_1|n}\sum_{q_3'\leq Q/q_1 }
\sum_{q_0\leq Q/(q_1q_3') \atop 4q_0 \, \mathrm{square-full}}
q_1^2q_0^2q_3'^{\frac{3}{2}}
 \left(\frac{1}{q_1q_0q_3'Q}+\frac{x}{(q_1q_0q_3'Q)^2} \right)\nonumber\\
&&+x^{\varepsilon}M\sum_{n\leq Q}
n^{-\frac{3}{2}+\varepsilon}
\sum_{q_1|n}\sum_{q_3'\leq Q/q_1 }
\sum_{q_0\leq Q/(q_1q_3') \atop 4q_0 \, \mathrm{square-full}}
q_1^2q_0^2q_3'^{\frac{3}{2}}\nonumber\\
&\ll_{\varepsilon}&\frac{x^{1+\varepsilon}}{Q}
\sum_{n\leq Q}
n^{-\frac{3}{2}+\varepsilon}
\sum_{q_1|n}q_1\sum_{q_3'\leq Q/q_1}q_3'^{\frac{1}{2}}
\left(\frac{Q}{q_1q_3'}\right)^{\frac{3}{2}}\nonumber\\
&&+\frac{x^{2+\varepsilon}}{Q^2}
\sum_{n\leq Q}
n^{-\frac{3}{2}+\varepsilon}
\sum_{q_1|n}\sum_{q_3'\leq Q/q_1}q_3'^{-\frac{1}{2}}
\left(\frac{Q}{q_1q_3'}\right)^{\frac{1}{2}}\nonumber\\
&&+x^{\varepsilon}M
\sum_{n\leq Q}
n^{-\frac{3}{2}+\varepsilon}
\sum_{q_1|n}q_1^2\sum_{q_3'\leq Q/q_1}q_3'^{\frac{3}{2}}
\left(\frac{Q}{q_1q_3'}\right)^{\frac{5}{2}}\nonumber\\
&\ll_{\varepsilon}&x^{1+\varepsilon}Q^{\frac{1}{2}}
+\frac{x^{2+\varepsilon}}{Q^{\frac{3}{2}}}+x^{\varepsilon}MQ^{\frac{5}{2}}\nonumber\\
&\ll_{\varepsilon}&Mx^{\frac{5}{4}+\varepsilon}.
\eea

Lastly, by (4.7), (4.15) and Propositions 5.1-5.2, we have
\bna
\mathscr{B}_4(v,q,\beta)
&\ll_{\varepsilon}&x^{\varepsilon}
\left(\frac{x}{1+|\beta|x}\right)^{\frac{3}{2}}
\sum_{n\leq Q}n^{-\frac{3}{2}}
q_1q_0q_3'^{\frac{1}{2}}\,X^{\varepsilon}n^{\varepsilon}(M+|\beta|^2X^2)\nonumber\\
&\ll_{\varepsilon}&\left(x^{\frac{3}{2}+\varepsilon}
\left(1+|\beta|x\right)^{\frac{1}{2}}+\frac{x^{\varepsilon}M}{(x^{-1}+|\beta|)^{\frac{3}{2}}}\right)
\sum_{n\leq Q}
n^{-\frac{3}{2}+\varepsilon}
q_1q_0q_3'^{\frac{1}{2}}.
\ena
By (4.3) and the estimate as above, we have
\bea
&&\sum_{q\leq Q}\int\limits_{|\beta|\leq \frac{1}{q Q}}\sum_{v\bmod q}
\varrho(v,q,\beta)\mathscr{B}_4(v,q,\beta)\mathrm{d}\beta\nonumber\\
&\ll_{\varepsilon}&x^{\frac{3}{2}+\varepsilon}
\sum_{n\leq Q}
n^{-\frac{3}{2}+\varepsilon}
\sum_{q_1|n}\sum_{q_3'\leq Q/q_1 \atop q_3' \,\mathrm{square-free}}
\sum_{q_0\leq Q/(q_1q_3') \atop 4q_0 \, \mathrm{square-full}}
q_1q_0q_3'^{\frac{1}{2}}
\int\limits_{|\beta|\leq \frac{1}{q_1q_0q_3'Q}}(1+|\beta|x)^{\frac{1}{2}}\mathrm{d}\beta\nonumber\\
&&+
x^{\varepsilon}M
\sum_{n\leq Q}
n^{-\frac{3}{2}+\varepsilon}
\sum_{q_1|n}\sum_{q_3'\leq Q/q_1 \atop q_3' \,\mathrm{square-free}}
\sum_{q_0\leq Q/(q_1q_3') \atop 4q_0 \, \mathrm{square-full}}
q_1q_0q_3'^{\frac{1}{2}}
\int\limits_{|\beta|\leq \frac{1}{q_1q_0q_3'Q}}(x^{-1}+|\beta|)^{-\frac{3}{2}}\mathrm{d}\beta
\nonumber
\eea
\bea
&\ll_{\varepsilon}&x^{\frac{3}{2}+\varepsilon}
\sum_{n\leq Q}
n^{-\frac{3}{2}+\varepsilon}
\sum_{q_1|n}\sum_{q_3'\leq Q/q_1 }
\sum_{q_0\leq Q/(q_1q_3') \atop 4q_0 \, \mathrm{square-full}} q_1q_0q_3'^{\frac{1}{2}}
 \left(\frac{1}{q_1q_0q_3'Q}+\frac{x^{\frac{1}{2}}}{(q_1q_0q_3'Q)^{\frac{3}{2}}} \right)\nonumber\\
&&+Mx^{\frac{1}{2}+\varepsilon}
\sum_{n\leq Q}
n^{-\frac{3}{2}+\varepsilon}
\sum_{q_1|n}\sum_{q_3'\leq Q/q_1 }
\sum_{q_0\leq Q/(q_1q_3') \atop 4q_0 \, \mathrm{square-full}} q_1q_0q_3'^{\frac{1}{2}}
\nonumber\\
&\ll_{\varepsilon}&\frac{x^{\frac{3}{2}+\varepsilon}}{Q}
\sum_{n\leq Q}
n^{-\frac{3}{2}+\varepsilon}
\sum_{q_1|n}\sum_{q_3'\leq Q/q_1}q_3'^{-\frac{1}{2}}
\left(\frac{Q}{q_1q_3'}\right)^{\frac{1}{2}}\nonumber\\
&&+\frac{x^{2+\varepsilon}}{Q^{\frac{3}{2}}}
\sum_{n\leq Q}
n^{-\frac{3}{2}+\varepsilon}
\sum_{q_1|n}q_1^{-\frac{1}{2}}\sum_{q_3'\leq Q/q_1}q_3'^{-1}\nonumber\\
&&+Mx^{\frac{1}{2}+\varepsilon}
\sum_{n\leq Q}
n^{-\frac{3}{2}+\varepsilon}
\sum_{q_1|n}q_1\sum_{q_3'\leq Q/q_1 }q_3'^{\frac{1}{2}}
\left(\frac{Q}{q_1q_3'}\right)^{\frac{3}{2}}
\nonumber\\
&\ll_{\varepsilon}&\frac{x^{\frac{3}{2}+\varepsilon}}{Q^{\frac{1}{2}}}
+\frac{x^{2+\varepsilon}}{Q^{\frac{3}{2}}}+Mx^{\frac{1}{2}+\varepsilon}Q^{\frac{3}{2}}\nonumber\\
&\ll_{\varepsilon}&Mx^{\frac{5}{4}+\varepsilon}.
\eea

By (4.2), (4.11) and (5.3)-(5.6), the contribution from
$\mathscr{B}_j$, $j=1,2,3,4$, is $O_{\varepsilon}(Mx^{\frac{5}{4}+\varepsilon})$.

\medskip

\section{Contribution of $\mathscr{B}_j$, $5\leq j\leq 13$}
\setcounter{equation}{0}
\medskip

First, we note that (recall (3.2) and (3.3))
\bea
P_j(n,q)\ll (\log (n+2)(q+2))^j, \qquad j=1,2
\eea
and
\bea
\widetilde{\phi_{\beta}}^{(j)}(1)=\int_0^{\infty}\phi\left(\frac{u}{X}\right)e(-\beta u)(\log u)^j\mathrm{d}u
\ll \frac{x(\log x)^j}{1+|\beta|x}, \qquad j=0,1,2.
\eea
Next, bounding the character sum $\mathscr{C}(b_1,b_2,b_3,n,m,v;q)$ by Weil's bound
for Kloosterman sums, we have
\bea
\mathscr{C}(b_1,b_2,b_3,n,m,v;q)\ll q^{\frac{5}{2}}\left(\frac{q}{n}\right)^{\frac{1}{2}}
\tau\left(\frac{q}{n}\right)\ll_{\varepsilon}q^{3+\varepsilon}n^{-\frac{1}{2}}.
\eea

By (4.7), (4.16)-(4.18) and (6.1)-(6.3), we have, for $j=5,6,7$,
\bea
\mathscr{B}_j(v,q,\beta)\ll_{\varepsilon}\frac{1}{q^2}\frac{x(\log x)^2}{1+|\beta|x}
\sum_{n|q}n\tau(n)q^{3+\varepsilon}n^{-\frac{1}{2}}(\log(q+2))^5
\ll_{\varepsilon}\frac{x^{\varepsilon}q^{\frac{3}{2}+\varepsilon}}{x^{-1}+|\beta|}.
\eea
By (4.3) and (6.4), we obtain, for $j=5,6,7$,
\bea
&&\sum_{q\leq Q}\int\limits_{|\beta|\leq \frac{1}{q Q}}\sum_{v\bmod q}
\varrho(v,q,\beta)\mathscr{B}_j(v,q,\beta)\mathrm{d}\beta\nonumber\\
&\ll_{\varepsilon}&x^{\varepsilon}\sum_{q\leq Q}q^{\frac{3}{2}+\varepsilon}\int\limits_{|\beta|\leq \frac{1}{q Q}}
\frac{1}{x^{-1}+|\beta|}\mathrm{d}\beta\nonumber\\
&\ll_{\varepsilon}&x^{\varepsilon}Q^{\frac{5}{2}+\varepsilon}.
\eea

By (4.7), (4.19)-(4.21) and (6.1)-(6.3), we have, for $j=8,9,10$,
\bea
\mathscr{B}_j(v,q,\beta)&\ll_{\varepsilon}&\frac{1}{q^3}\frac{x(\log x)^2}{1+|\beta|x}
\left(\frac{x}{1+|\beta|x}\right)^{\frac{1}{2}}
\sum_{n|q}n\tau(n)q^{3+\varepsilon}n^{-\frac{1}{2}}(\log(q+2))^4\nonumber\\
&\ll_{\varepsilon}&x^{\varepsilon}q^{\frac{1}{2}+\varepsilon}
\left(\frac{1}{x^{-1}+|\beta|}\right)^{\frac{3}{2}}.
\eea
By (4.3) and (6.6), we obtain, for $j=8,9,10$,
\bea
&&\sum_{q\leq Q}\int\limits_{|\beta|\leq \frac{1}{q Q}}\sum_{v\bmod q}
\varrho(v,q,\beta)\mathscr{B}_j(v,q,\beta)\mathrm{d}\beta\nonumber\\
&\ll_{\varepsilon}&x^{\varepsilon}\sum_{q\leq Q}q^{\frac{1}{2}+\varepsilon}
\int\limits_{|\beta|\leq \frac{1}{q Q}}
\left(\frac{1}{x^{-1}+|\beta|}\right)^{\frac{3}{2}}\mathrm{d}\beta\nonumber\\
&\ll_{\varepsilon}&x^{\frac{1}{2}+\varepsilon}Q^{\frac{3}{2}+\varepsilon}.
\eea

By (4.7), (4.22)-(4.24) and (6.1)-(6.3), we have, for $j=11,12,13$,
\bea
\mathscr{B}_j(v,q,\beta)&\ll_{\varepsilon}&\frac{1}{q^4}\frac{x(\log x)^2}{1+|\beta|x}
\frac{x}{1+|\beta|x}
\sum_{n|q}n\tau(n)q^{3+\varepsilon}n^{-\frac{1}{2}}(\log(q+2))^3\nonumber\\
&\ll_{\varepsilon}&x^{\varepsilon}q^{-\frac{1}{2}+\varepsilon}
\left(\frac{1}{x^{-1}+|\beta|}\right)^2.
\eea
By (4.3) and (6.8), we obtain, for $j=11,12,13$,
\bea
&&\sum_{q\leq Q}\int\limits_{|\beta|\leq \frac{1}{q Q}}\sum_{v\bmod q}
\varrho(v,q,\beta)\mathscr{B}_j(v,q,\beta)\mathrm{d}\beta\nonumber\\
&\ll_{\varepsilon}&x^{\varepsilon}\sum_{q\leq Q}q^{-\frac{1}{2}+\varepsilon}\int\limits_{|\beta|\leq \frac{1}{q Q}}
\left(\frac{1}{x^{-1}+|\beta|}\right)^2\mathrm{d}\beta\nonumber\\
&\ll_{\varepsilon}&x^{1+\varepsilon}Q^{\frac{1}{2}+\varepsilon}.
\eea

By (4.2), (4.11), (6.5), (6.7) and (6.9), the contribution from
$\mathscr{B}_j$, $5\leq j\leq 13$, is $O_{\varepsilon}(x^{\frac{5}{4}+\varepsilon})$.

\medskip

\section{Computation of the main terms}
\setcounter{equation}{0}
\medskip

The three sums $\mathscr{B}_{14}(v,q,\beta)$, $\mathscr{B}_{15}(v,q,\beta)$
and $\mathscr{B}_{16}(v,q,\beta)$ in (4.25)-(4.27) contribute the main terms.
Replacing $\widetilde{\phi_{\beta}}^{(j)}(1)$ by
\bea
\mathcal{\vartheta}^{\flat,j}(\beta)=\int_{X/2}^Xe(-\beta u)(\log u)^j\mathrm{d}u, \qquad j=0,1,2,
\eea
we need to estimate the remainder terms from
\bna
\vartheta^{\sharp,j}(\beta)=\widetilde{\phi_{\beta}}^{(j)}(1)-\mathcal{\vartheta}^{\flat,j}(\beta).
\ena
Write correspondingly
\bna
\mathscr{B}_j^{\sharp}(v,q,\beta)=\mathscr{B}_j(v,q,\beta)-\mathscr{B}_j^{\flat}(v,q,\beta),
\qquad
j=14,15,16.
\ena

First, we evaluate the remainder terms from $\mathscr{B}_j^{\sharp}(v,q,\beta)$, $j=14,15,16$. Notice that
\bna
\vartheta^{\sharp,j}(\beta)=
\int_{X/2}^X\left(\phi\left(\frac{u}{X}\right)-1\right)e(-\beta u)(\log u)^j\mathrm{d}u
\ll X M^{-1}(\log X)^j.
\ena
Hence
\bea
&&\sum_{q\leq Q}\int\limits_{|\beta|\leq \frac{1}{q Q}}\sum_{v\bmod q}
\varrho(v,q,\beta)\mathscr{B}_{14}^{\sharp}(v,q,\beta)\mathrm{d}\beta\nonumber\\
&=&\frac{1}{2}\sum_{q\leq Q}\frac{1}{q^5}\int\limits_{|\beta|\leq \frac{1}{q Q}}
\sum_{v\bmod q}\varrho(v,q,\beta)\vartheta^{\sharp,0}(\beta)\Psi_0(\beta)^3
\nonumber\\
&&\times \sum_{n|q}n\tau(n)P_2(n,q)\mathscr{C}(0,0,0,n,0,v;q)\mathrm{d}\beta\nonumber\\
&\ll_{\varepsilon}&X^{1+\varepsilon}M^{-1}\sum_{q\leq Q}q^{-2+\varepsilon}
\sum_{n|q}n^{\frac{1}{2}}\tau(n)|P_2(n,q)|
\int\limits_{|\beta|\leq \frac{1}{q Q}}
\left(\frac{1}{x^{-1}+|\beta|}\right)^{\frac{3}{2}}\mathrm{d}\beta\nonumber\\
&\ll_{\varepsilon}&x^{\frac{3}{2}+\varepsilon}M^{-1}\sum_{q\leq Q}q^{-2+\varepsilon}
\sum_{n|q}n^{\frac{1}{2}}\tau(n)|P_2(n,q)|\nonumber\\
&\ll_{\varepsilon}&x^{\frac{3}{2}+\varepsilon}M^{-1}.
\eea
Here we have used (4.3), (5.2) and (6.3).
Similarly,
\bea
&&\sum_{q\leq Q}\int\limits_{|\beta|\leq \frac{1}{q Q}}\sum_{v\bmod q}
\varrho(v,q,\beta)\mathscr{B}_{15}^{\sharp}(v,q,\beta)\mathrm{d}\beta\nonumber\\
&=&\frac{1}{2}\sum_{q\leq Q}\frac{1}{q^5}\int\limits_{|\beta|\leq \frac{1}{q Q}}
\sum_{v\bmod q}\varrho(v,q,\beta)\vartheta^{\sharp,1}(\beta)\Psi_0(\beta)^3
\nonumber\\
&&\times \sum_{n|q}n\tau(n)P_1(n,q)\mathscr{C}(0,0,0,n,0,v;q)\mathrm{d}\beta\nonumber\\
&\ll_{\varepsilon}&X^{1+\varepsilon}M^{-1}\sum_{q\leq Q}q^{-2+\varepsilon}
\sum_{n|q}n^{\frac{1}{2}}\tau(n)|P_1(n,q)|
\int\limits_{|\beta|\leq \frac{1}{q Q}}
\left(\frac{1}{x^{-1}+|\beta|}\right)^{\frac{3}{2}}\mathrm{d}\beta\nonumber\\
&\ll_{\varepsilon}&x^{\frac{3}{2}+\varepsilon}M^{-1}\sum_{q\leq Q}q^{-2+\varepsilon}
\sum_{n|q}n^{\frac{1}{2}}\tau(n)|P_1(n,q)|\nonumber\\
&\ll_{\varepsilon}&x^{\frac{3}{2}+\varepsilon}M^{-1},
\eea
and
\bea
&&\sum_{q\leq Q}\int\limits_{|\beta|\leq \frac{1}{q Q}}\sum_{v\bmod q}
\varrho(v,q,\beta)\mathscr{B}_{16}^{\sharp}(v,q,\beta)\mathrm{d}\beta\nonumber\\
&=&\frac{1}{4}\sum_{q\leq Q}\frac{1}{q^5}\int\limits_{|\beta|\leq \frac{1}{q Q}}
\sum_{v\bmod q}\varrho(v,q,\beta)\vartheta^{\sharp,2}(\beta)\Psi_0(\beta)^3
\sum_{n|q}n\tau(n)\mathscr{C}(0,0,0,n,0,v;q)\mathrm{d}\beta\nonumber\\
&\ll_{\varepsilon}&X^{1+\varepsilon}M^{-1}\sum_{q\leq Q}q^{-2+\varepsilon}
\sum_{n|q}n^{\frac{1}{2}}\tau(n)
\int\limits_{|\beta|\leq \frac{1}{q Q}}
\left(\frac{1}{x^{-1}+|\beta|}\right)^{\frac{3}{2}}\mathrm{d}\beta\nonumber\\
&\ll_{\varepsilon}&x^{\frac{3}{2}+\varepsilon}M^{-1}\sum_{q\leq Q}q^{-2+\varepsilon}
\sum_{n|q}n^{\frac{1}{2}}\tau(n)\nonumber\\
&\ll_{\varepsilon}&x^{\frac{3}{2}+\varepsilon}M^{-1}.
\eea

Next, we want to compute the contributions from
$\mathscr{B}_j^{\flat}(v,q,\beta)$ which constitute the main terms.
Interchanging
the order of summation over $a$ and the integration over $\beta$, we have
\bna
&&\sum_{q\leq Q}\int\limits_{|\beta|\leq \frac{1}{q Q}}\sum_{v\bmod q}
\varrho(v,q,\beta)\mathscr{B}_{14}^{\flat}(v,q,\beta)\mathrm{d}\beta\nonumber\\
&=&\frac{1}{2}\sum_{q\leq Q}\frac{1}{q^5}\int\limits_{|\beta|\leq \frac{1}{q Q}}
\sum_{v\bmod q}\varrho(v,q,\beta)\vartheta^{\flat,0}(\beta)\Psi_0(\beta)^3
\sum_{n|q}n\tau(n)P_2(n,q)\nonumber\\
&&\times\sideset{}{^*}\sum_{a=1}^q
e\left(\frac{-\overline{a}v}{q}\right)
G(a,0;q)^3 S\left(-\overline{a},0;\frac{q}{n}\right)\mathrm{d}\beta\nonumber\\
&=&\frac{1}{2}\sum_{q\leq Q}\frac{1}{q^5}\sum_{n|q}n\tau(n)P_2(n,q)
\sideset{}{^*}\sum_{a=1}^q
G(a,0;q)^3 S\left(-\overline{a},0;\frac{q}{n}\right)
\int\limits_{\mathscr{M}(a,q)}
\vartheta^{\flat,0}(\beta)\Psi_0(\beta)^3\mathrm{d}\beta.
\ena
Note that
\bna
\left[-\frac{1}{2qQ},\frac{1}{2qQ}\right]\subseteq\mathscr{M}(a,q)\subseteq
\left[-\frac{1}{q Q}, \frac{1}{q Q}\right].
\ena
As in \cite{Z}, we write $\mathscr{M}(a,q)$ as
\bna
\mathscr{M}(a,q)=\mathscr{M}(a,q)\backslash \left[-\frac{1}{2qQ},\frac{1}{2qQ}\right]
\bigcup \left[-\frac{1}{2qQ},\frac{1}{2qQ}\right].
\ena
Accordingly,
\bea
&&\sum_{q\leq Q}\int\limits_{|\beta|\leq \frac{1}{q Q}}\sum_{v\bmod q}
\varrho(v,q,\beta)\mathscr{B}_{14}^{\flat}(v,q,\beta)\mathrm{d}\beta\nonumber\\
&=&\frac{1}{2}\sum_{q\leq Q}\frac{1}{q^5}\sum_{n|q}n\tau(n)P_2(n,q)
\sideset{}{^*}\sum_{a=1}^q
G(a,0;q)^3 S\left(-\overline{a},0;\frac{q}{n}\right)
\int\limits_{-\infty}^{\infty}
\vartheta^{\flat,0}(\beta)\Psi_0(\beta)^3\mathrm{d}\beta\nonumber\\
&&+\mathscr{B}_{14}^*-\mathscr{B}_{14}^{**},
\eea
where
\bna
\mathscr{B}_{14}^*&=&\frac{1}{2}\sum_{q\leq Q}\frac{1}{q^5}\sum_{n|q}n\tau(n)P_2(n,q)
\sideset{}{^*}\sum_{a=1}^q
G(a,0;q)^3 S\left(-\overline{a},0;\frac{q}{n}\right)\nonumber\\
&&\times\int\limits_{\mathscr{M}(a,q)\backslash \left[-\frac{1}{2qQ},\frac{1}{2qQ}\right]}
\vartheta^{\flat,0}(\beta)\Psi_0(\beta)^3\mathrm{d}\beta\nonumber\\
\mathscr{B}_{14}^{**}&=&\frac{1}{2}\sum_{q\leq Q}\frac{1}{q^5}\sum_{n|q}n\tau(n)P_2(n,q)
\sideset{}{^*}\sum_{a=1}^q
G(a,0;q)^3 S\left(-\overline{a},0;\frac{q}{n}\right)
\int\limits_{|\beta|>\frac{1}{2qQ}}
\vartheta^{\flat,0}(\beta)\Psi_0(\beta)^3\mathrm{d}\beta.
\ena
Interchanging
the order of summation over $a$ and the integration over $\beta$ again, we have
\bea
\mathscr{B}_{14}^*
&=&\frac{1}{2}\sum_{q\leq Q}\frac{1}{q^5}\sum_{n|q}n\tau(n)P_2(n,q)
\int\limits_{\frac{1}{2qQ}\leq |\beta|\leq \frac{1}{qQ}}
\sum_{v\bmod q}
\varrho(v,q,\beta)
\vartheta^{\flat,0}(\beta)\Psi_0(\beta)^3
\nonumber\\
&&\sideset{}{^*}\sum_{a=1}^qe\left(\frac{-\overline{a}v}{q}\right)
G(a,0;q)^3 S\left(-\overline{a},0;\frac{q}{n}\right)
\mathrm{d}\beta\nonumber\\
&\ll_{\varepsilon}&\sum_{q\leq Q}q^{-2+\varepsilon}\sum_{n|q}n^{\frac{1}{2}}\tau(n)|P_2(n,q)|
\int\limits_{\frac{1}{2qQ}\leq |\beta|\leq \frac{1}{qQ}}|\beta|^{-\frac{5}{2}}\mathrm{d}\beta
\nonumber\\
&\ll_{\varepsilon}&Q^{\frac{5}{2}+\varepsilon},
\eea
and
\bea
\mathscr{B}_{14}^{**}
&=&\frac{1}{2}\sum_{q\leq Q}\frac{1}{q^5}\sum_{n|q}n\tau(n)P_2(n,q)
\int\limits_{|\beta|> \frac{1}{2qQ}}
\sum_{v\bmod q}
\varrho(v,q,\beta)
\vartheta^{\flat,0}(\beta)\Psi_0(\beta)^3
\nonumber\\
&&\sideset{}{^*}\sum_{a=1}^qe\left(\frac{-\overline{a}v}{q}\right)
G(a,0;q)^3 S\left(-\overline{a},0;\frac{q}{n}\right)
\mathrm{d}\beta\nonumber\\
&\ll_{\varepsilon}&\sum_{q\leq Q}q^{-2+\varepsilon}\sum_{n|q}n^{\frac{1}{2}}\tau(n)|P_2(n,q)|
\int\limits_{|\beta|>\frac{1}{2qQ}}|\beta|^{-\frac{5}{2}}\mathrm{d}\beta
\nonumber\\
&\ll_{\varepsilon}&Q^{\frac{5}{2}+\varepsilon}.
\eea
Here we have used (4.3), (5.2) and (6.1)-(6.3).
By (7.5)-(7.7), we obtain
\bea
&&\sum_{q\leq Q}\int\limits_{|\beta|\leq \frac{1}{q Q}}\sum_{v\bmod q}
\varrho(v,q,\beta)\mathscr{B}_{14}^{\flat}(v,q,\beta)\mathrm{d}\beta\nonumber\\
&=&\frac{1}{2}\sum_{q\leq Q}\frac{1}{q^5}\sum_{n|q}n\tau(n)P_2(n,q)
\sideset{}{^*}\sum_{a=1}^q
G(a,0;q)^3 S\left(-\overline{a},0;\frac{q}{n}\right)\nonumber\\
&&\times\int\limits_{-\infty}^{\infty}
\vartheta^{\flat,0}(\beta)\Psi_0(\beta)^3\mathrm{d}\beta
+O_{\varepsilon}\left(x^{\frac{5}{4}+\varepsilon}\right).
\eea
Moreover, by (4.6),
\bna
\Psi_0(\beta)=\int_0^{\sqrt{x}}e(\beta v^2)\mathrm{d}v=
x^{\frac{1}{2}}\int_0^{1}e(\beta x v^2)\mathrm{d}v.
\ena
It follows that
\bna
\int\limits_{-\infty}^{\infty}
\vartheta^{\flat,0}(\beta)\Psi_0(\beta)^3\mathrm{d}\beta
=x^{\frac{3}{2}}\int_{-\infty}^{\infty}
\left(\int_{X/2}^{X}e(-\beta u)\mathrm{d}u\right)
\left(\int_0^{1}e(\beta  x v^2)\mathrm{d}v\right)^3\mathrm{d}\beta
:=x^{\frac{3}{2}}\mathcal{I}_0(X),\nonumber\\
\ena
where
\bea
\mathcal{I}_0(X)=\int_{-\infty}^{\infty}
\left(\int_{X/2}^{X}e(-\beta u)\mathrm{d}u\right)
\left(\int_0^{1}e(\beta x v^2)\mathrm{d}v\right)^3\mathrm{d}\beta.
\eea
Substituting in (7.8) and by (7.2), we have
\bea
&&\sum_{q\leq Q}\int\limits_{|\beta|\leq \frac{1}{q Q}}\sum_{v\bmod q}
\varrho(v,q,\beta)\mathscr{B}_{14}(v,q,\beta)\mathrm{d}\beta\nonumber\\
&=&\frac{1}{2}\mathcal{I}_0(X)x^{\frac{3}{2}}\sum_{q\leq Q}\frac{1}{q^5}\sum_{n|q}n\tau(n)P_2(n,q)
\mathscr{C}(0,0,0,n,0,0;q)
+O_{\varepsilon}\left(x^{\frac{5}{4}+\varepsilon}+x^{\frac{3}{2}+\varepsilon}M^{-1}\right).\nonumber\\
\eea
Further, by (6.1) and (6.3), we have
\bea
&&\sum_{q>Q}\frac{1}{q^5}\sum_{n|q}n\tau(n)P_2(n,q)
\mathscr{C}(0,0,0,n,0,0;q)\nonumber\\
&\ll_{\varepsilon}&\sum_{q>Q}q^{-2+\varepsilon}\sum_{n|q}n^{\frac{1}{2}}\tau(n)|P_2(n,q)|\nonumber\\
&\ll_{\varepsilon}&Q^{-\frac{1}{2}+\varepsilon}.
\eea
By (7.10) and (7.11), we conclude that
\bea
\sum_{q\leq Q}\int\limits_{|\beta|\leq \frac{1}{q Q}}\sum_{v\bmod q}
\varrho(v,q,\beta)\mathscr{B}_{14}(v,q,\beta)\mathrm{d}\beta
=\frac{1}{2}\mathcal{I}_0(X)\mathcal {C}_2x^{\frac{3}{2}}
+O_{\varepsilon}\left(x^{\frac{5}{4}+\varepsilon}+x^{\frac{3}{2}+\varepsilon}M^{-1}\right),
\eea
where
\bea
\mathcal {C}_2=\sum_{q=1}^{\infty}\frac{1}{q^5}\sum_{n|q}n\tau(n)P_2(n,q)
\sideset{}{^*}\sum_{a=1}^q
G(a,0;q)^3 S\left(-\overline{a},0;\frac{q}{n}\right).
\eea

Similarly, we have
\bea
\sum_{q\leq Q}\int\limits_{|\beta|\leq \frac{1}{q Q}}\sum_{v\bmod q}
\varrho(v,q,\beta)\mathscr{B}_{15}^{\flat}(v,q,\beta)\mathrm{d}\beta
=\frac{1}{2}\mathcal{I}_1(X)\mathcal {C}_1x^{\frac{3}{2}}
+O_{\varepsilon}\left(x^{\frac{5}{4}+\varepsilon}\right),
\eea
where
\bea
\mathcal{I}_1(X)=\int_{-\infty}^{\infty}
\left(\int_{X/2}^Xe(-\beta u)(\log u)\mathrm{d}u\right)
\left(\int_0^{1}e(\beta  x v^2)\mathrm{d}v\right)^3\mathrm{d}\beta,
\eea
and
\bea
\mathcal {C}_1=\sum_{q=1}^{\infty}\frac{1}{q^5}\sum_{n|q}n\tau(n)P_1(n,q)
\sideset{}{^*}\sum_{a=1}^q
G(a,0;q)^3 S\left(-\overline{a},0;\frac{q}{n}\right).
\eea
By (7.3) and (7.14), we have
\bea
\sum_{q\leq Q}\int\limits_{|\beta|\leq \frac{1}{q Q}}\sum_{v\bmod q}
\varrho(v,q,\beta)\mathscr{B}_{15}(v,q,\beta)\mathrm{d}\beta
=\frac{1}{2}\mathcal{I}_1(X)\mathcal {C}_1x^{\frac{3}{2}}
+O_{\varepsilon}\left(x^{\frac{5}{4}+\varepsilon}+x^{\frac{3}{2}+\varepsilon}M^{-1}\right).
\eea

Finally, we have
\bea
\sum_{q\leq Q}\int\limits_{|\beta|\leq \frac{1}{q Q}}\sum_{v\bmod q}
\varrho(v,q,\beta)\mathscr{B}_{16}^{\flat}(v,q,\beta)\mathrm{d}\beta
=\frac{1}{4}\mathcal{I}_2(X)\mathcal {C}_0x^{\frac{3}{2}}
+O_{\varepsilon}\left(x^{\frac{5}{4}+\varepsilon}\right),
\eea
where
\bea
\mathcal{I}_2(X)=\int_{-\infty}^{\infty}
\left(\int_{X/2}^Xe(-\beta u)(\log u)^2\mathrm{d}u\right)
\left(\int_0^{1}e(\beta x v^2)\mathrm{d}v\right)^3\mathrm{d}\beta,
\eea
and
\bea
\mathcal {C}_0=\sum_{q=1}^{\infty}\frac{1}{q^5}\sum_{n|q}n\tau(n)
\sideset{}{^*}\sum_{a=1}^q
G(a,0;q)^3 S\left(-\overline{a},0;\frac{q}{n}\right).
\eea
By (7.4) and (7.18), we obtain
\bea
\sum_{q\leq Q}\int\limits_{|\beta|\leq \frac{1}{q Q}}\sum_{v\bmod q}
\varrho(v,q,\beta)\mathscr{B}_{16}(v,q,\beta)\mathrm{d}\beta
=\frac{1}{4}\mathcal{I}_2(X)\mathcal {C}_0X^{\frac{3}{2}}
+O_{\varepsilon}\left(x^{\frac{5}{4}+\varepsilon}+x^{\frac{3}{2}+\varepsilon}M^{-1}\right).
\eea

By (4.2), (4.11), (7.12), (7.17) and (7.21), the contribution from
$\mathscr{B}_j$, $14\leq j\leq 16$, is
\bna
\frac{1}{2}\mathcal{I}_0(X)\mathcal {C}_2x^{\frac{3}{2}}+\frac{1}{2}\mathcal{I}_1(X)\mathcal {C}_1x^{\frac{3}{2}}+
\frac{1}{4}\mathcal{I}_2(X)\mathcal {C}_0x^{\frac{3}{2}}+O_{\varepsilon}\left(x^{\frac{5}{4}+\varepsilon}+
x^{\frac{3}{2}+\varepsilon}M^{-1}\right),
\ena
where $\mathcal{I}_j$ ($0\leq j\leq 2$) and $\mathcal {C}_j$ ($0\leq j\leq 2$)
are defined in (7.9), (7.15), (7.19) and (7.13), (7.16), (7.20),
respectively.

\medskip

\section{Proof of Proposition 5.2}
\setcounter{equation}{0}
\medskip

Recall $\Phi_{\beta}^{\pm}(y)$ in (4.9) which we relabel as
\bea
\Phi_{\beta}^{\pm}(y)=\Phi_0(y,\beta)\pm\frac{1}{i\pi^3y}\Phi_1(y,\beta),
\eea
where for $\sigma>-1-k$,
\bea
\Phi_k(y,\beta)=(\pi^3y)^k\frac{1}{2\pi i}\int\limits_{\mathrm{Re}(s)=\sigma}
(\pi^3 y)^{-s}
\frac{\Gamma\left(\frac{1+s+k}{2}\right)^3}
{\Gamma\left(\frac{-s+k}{2}\right)^3}
\widetilde{\phi_{\beta}}(-s)\mathrm{d}s
\eea
with $\phi_{\beta}(y)=\phi\left(\frac{y}{X}\right)e(-\beta y)$.
Note that
\bna
\phi_{\beta}^{(j)}(y)\ll_j \left(\frac{M+|\beta|X}{X}\right)^j.
\ena
By Lemma 3.2, we have
\bea
&&\sum_{\pm}\sum_{m\geq 1}\frac{1}{m}\sum_{n_1|n}\sum_{n_2|\frac{n}{n_1}}
\sigma_{0,0}\left(\frac{n}{n_1n_2},m\right)
\left|\Phi_{\beta}^{\pm}\left(\frac{mn^2}{q^3}\right)\right|\nonumber\\
&&=
\sum_{\pm}\sum_{\frac{mn^2}{q^3}X< X^{\varepsilon}(M+|\beta|X)^3}\frac{1}{m}\sum_{n_1|n}\sum_{n_2|\frac{n}{n_1}}
\sigma_{0,0}\left(\frac{n}{n_1n_2},m\right)
\left|\Phi_{\beta}^{\pm}\left(\frac{mn^2}{q^3}\right)\right|+O_{\varepsilon}(1).
\eea
Moreover, by (3.1), trivially, we have
\bea
\sum_{m\leq L}\sum_{n_1|n}\sum_{n_2|\frac{n}{n_1}}
\sigma_{0,0}\left(\frac{n}{n_1n_2},m\right)\ll
\sum_{m\leq L}\tau_3(n)\tau_3(m)\ll_{\varepsilon} n^{\varepsilon}L^{1+\varepsilon}.
\eea

For $yX\ll X^{\varepsilon}$,
we move the line of integration in (8.2) to $\sigma=-1+\varepsilon$ to obtain
\bea
\Phi_k(y,\beta)&=&(\pi^3y)^k\frac{1}{2\pi i}\int\limits_{\mathrm{Re}(s)=-1+\varepsilon}
(\pi^3 y)^{-s}
\frac{\Gamma\left(\frac{1+s+k}{2}\right)^3}
{\Gamma\left(\frac{-s+k}{2}\right)^3}
\widetilde{\phi_{\beta}}(-s)\mathrm{d}s\nonumber\\
&\ll&y^k(y X)^{1-\varepsilon}\int_{-\infty}^{\infty}(1+|t|)^{-\frac{3}{2}+3\varepsilon}\mathrm{d}t
\nonumber\\
&\ll_{\varepsilon}&y^kX^{\varepsilon}.
\eea
Here we have used Stirling's formula and the estimate
$\widetilde{\phi_{\beta}}(-s)=\int_0^{\infty}\phi_{\beta}(u)u^{-s-1}\mathrm{d}u\ll X^{-\sigma}$.
By (8.1), (8.4) and (8.5), we have
\bea
&&\sum_{\pm}\sum_{\frac{mn^2}{q^3}X\ll X^{\varepsilon}}\frac{1}{m}\sum_{n_1|n}\sum_{n_2|\frac{n}{n_1}}
\sigma_{0,0}\left(\frac{n}{n_1n_2},m\right)
\left|\Phi_{\beta}^{\pm}\left(\frac{mn^2}{q^3}\right)\right|\nonumber\\
&\ll_{\varepsilon}&X^{\varepsilon}\max_{1\leq L\ll \frac{q^3X^{\varepsilon}}{n^2X}}\frac{1}{L}
\sum_{L<m\leq 2L}\sum_{n_1|n}\sum_{n_2|\frac{n}{n_1}}
\sigma_{0,0}\left(\frac{n}{n_1n_2},m\right)\nonumber\\
&\ll_{\varepsilon}&X^{\varepsilon}
\max_{1\leq L\ll \frac{q^3X^{\varepsilon}}{n^2X}}\frac{1}{L} n^{\varepsilon}L^{1+\varepsilon}\nonumber\\
&\ll_{\varepsilon}&X^{\varepsilon}n^{\varepsilon}.
\eea

For $yX>X^{\varepsilon}$, by Lemma 3.3, we have
\bna
\Phi_k(y,\beta)&=&(\pi^3 y)^{k+1}\sum_{j=1}^{\ell} \int_0^{\infty}\phi\left(\frac{u}{X}\right)e(-\beta u)
\left(a_k(j)e\left(3(y u)^{\frac{1}{3}}\right)+b_k(j)e\left(-3(y u)^{\frac{1}{3}}\right)\right)\\
&&\times\frac{\mathrm{d}u}{(\pi^3yu)^{\frac{j}{3}}}
+O_{\varepsilon,\ell}\left((\pi^3y)^k(\pi^3yX)^{-\frac{\ell}{3}+\frac{1}{2}+\varepsilon}\right),
\ena
where $a_k(j)$ and $b_k(j)$ are constants. Then by (8.1),
\bea
\Phi_{\beta}^{\pm}(y)\ll_{\varepsilon,\ell}y\sum_{j=1}^{\ell} y^{-\frac{j}{3}}\left(|\mathscr{I}_j(y,\beta)|
+|\mathscr{J}_j(y,\beta)|\right)+(y X)^{-\frac{\ell}{3}+\frac{1}{2}+\varepsilon},
\eea
with
\bna
\mathscr{I}_j(y,\beta)&=&\int_0^{\infty}u^{-\frac{j}{3}}\phi\left(\frac{u}{X}\right)e(-\beta u)
e\left(3(y u)^{\frac{1}{3}}\right)\mathrm{d}u,\\
\mathscr{J}_j(y,\beta)&=&\int_0^{\infty}u^{-\frac{j}{3}}\phi\left(\frac{u}{X}\right)e(-\beta u)
e\left(-3(y u)^{\frac{1}{3}}\right)\mathrm{d}u.
\ena

By partial integration twice, we have
\bea
\mathscr{I}_j(y,\beta), \mathscr{J}_j(y,\beta)
\ll (y X)^{-\frac{2}{3}}X^{1-\frac{j}{3}}\left(M+|\beta|^2X^2\right).
\eea
Taking $\ell=3$. By (8.7) and (8.8), we have
\bea
\Phi_{\beta}^{\pm}(y)\ll_{\varepsilon}(y X)^{\frac{1}{3}}(M+|\beta|^2X^2)
\sum_{j=1}^3 (y X)^{-\frac{j}{3}}+(y X)^{-\frac{1}{2}+\varepsilon}
\ll_{\varepsilon} M+|\beta|^2X^2.
\eea
By (8.4) and (8.9), we have
\bea
&&\sum_{\pm}\sum_{X^{\varepsilon}<\frac{n^2m}{q^3}X< X^{\varepsilon}(M+|\beta|X)^3}\frac{1}{m}\sum_{n_1|n}\sum_{n_2|\frac{n}{n_1}}
\sigma_{0,0}\left(\frac{n}{n_1n_2},m\right)
\left|\Phi_{\beta}^{\pm}\left(\frac{mn^2}{q^3}\right)\right|\nonumber\\
&\ll_{\varepsilon}&(M+|\beta|^2X^2) (\log X)
\max_{\frac{q^3X^{\varepsilon}}{n^2X}<L< \frac{q^3X^{\varepsilon}(M+|\beta|X)^3}{n^2X}}\frac{1}{L}
\sum_{L<m\leq 2L}\sum_{n_1|n}\sum_{n_2|\frac{n}{n_1}}
\sigma_{0,0}\left(\frac{n}{n_1n_2},m\right)\nonumber\\
&\ll_{\varepsilon}&X^{\varepsilon}n^{\varepsilon}(M+|\beta|^2X^2).
\eea
Then Proposition 5.2 follows from (8.3), (8.6) and (8.10).

\medskip

\section{Estimation of the character sum $\mathscr{C}(b_1,b_2,b_3,n,m,v;q)$}
\setcounter{equation}{0}
\medskip

In this section, we shall prove Proposition 5.1. Let $b_1,b_2,b_3,n, m, v\in \mathbb{Z}$ and $n|q$.
We recall $\mathscr{C}(b_1,b_2,b_3,n,m,v;q)$ in (4.28) which we relabel as
\bea
\mathscr{C}(b_1,b_2,b_3,n,m,v;q)=\sideset{}{^*}\sum_{a\bmod q}e\left(\frac{-\overline{a}v}{q}\right)
G(a,b_1;q)G(a,b_2;q)G(a,b_3;q) S\left(-\overline{a},m;\frac{q}{n}\right),
\eea
where $G(a,b;q)$ is the Gauss sum defined in (4.5).
Here for notation simplicity, we have replaced $\pm m$ in (4.28) by $m$, which does not affect our argument.

We first list some well-known results for $G(a,b;q)$
(see for example Lemma 5.4.5 in \cite{Huxley}).
For $(a,q)=1$, we have
\bea
G(a,b;q)\ll \sqrt{q}.
\eea
For $(2a,q)=1$, we have
\bea
G(a,b;q)=
e\left(-\frac{\bar{4}\bar{a}b^2}{q}\right)G(a,0;q)
\eea
and
\bea
G(a,0;q)=\left(\frac{a}{q}\right)\epsilon_q\sqrt{q},
\eea
where
$
\epsilon_q=\left\{\begin{array}{ll}
1,&\mbox{if $q\equiv 1(\bmod 4)$},\\
i,&\mbox{if $q\equiv -1(\bmod 4)$.}
\end{array}
\right.
$

As in \cite{Sun}, we factor $q$ as $q=q_1q_2q_3$, where $q_1$ is the largest factor of $q$ such that $q_1|n$, $(q_1,q_2q_3)=1$,
$q_2|n^{\infty}$, $(q_2,q_3)=1$. Then $q_2$ is square-full and $n|q_1q_2$. Denote
temporarily $q'=q_1q_2$ and $\widehat{q}=\frac{q'}{n}$.
Then we have
\bea
&&\mathscr{C}(b_1,b_2,b_3,n,m,v;q)\nonumber\\
&=&\sideset{}{^*}\sum_{a_1\bmod q'}
e\left(\frac{-\overline{a_1}v}{q'}\right)
G(a_1,b_1;q')G(a_1,b_2;q')G(a_1,b_3;q')
S(-\overline{a_1}q_3,m\overline{q_3}^2;\widehat{q})\nonumber\\
&&\times \sideset{}{^*}\sum_{a_2\bmod q_3}
e\left(\frac{-\overline{a}_2v}{q_3}\right)
G(a_2,b_1;q_3)G(a_2,b_2;q_3)G(a_2,b_3;q_3)
S(-\overline{a_2}q',m\overline{\widehat{q}}^2;q_3)
\nonumber\\
&:=&\mathscr{C}^*(b_1,b_2,b_3,n,m,v;q')\mathscr{C}^{**}(b_1,b_2,b_3,n,m,v;q_3)
\eea
say.

By (9.2) and Weil's bound for Kloosterman sum we have
\bea
\mathscr{C}^*(b_1,b_2,b_3,n,m,v;q')\ll q'^{\frac{5}{2}}
\left(\overline{a_1}q_3,m\overline{q_3}^2,\frac{q'}{n}\right)^{\frac{1}{2}}
\left(\frac{q'}{n}\right)^{\frac{1}{2}}\tau\left(\frac{q'}{n}\right)
\ll \frac{q_1^3q_2^3\tau(q_1q_2)}{\sqrt{n}}.
\eea

To estimate $\mathscr{C}^{**}(b_1,b_2,b_3,n,m,v;q_3)$, we further factor
$q_3$ as $q_3=q_3'q_3''$ with $(q_3',2q_3'')=1$, $q_3'$ square-free and $4q_3''$ square-full.
Then
\bea
\mathscr{C}^{**}(b_1,b_2,b_3,n,m,v;q_3)=\mathscr{C}_1^{**}(b_1,b_2,b_3,n,m,v;q_3')
\mathscr{C}_2^{**}(b_1,b_2,b_3,n,m,v;q_3''),
\eea
where
\bna
\mathscr{C}_1^{**}(b_1,b_2,b_3,n,m,v;q_3')&=&\sideset{}{^*}\sum_{\gamma \bmod q_3'}
e\left(\frac{-\overline{\gamma}v}{q_3'}\right)
G(\gamma,b_1;q_3')G(\gamma,b_2;q_3') G(\gamma,b_3;q_3')\\
&&\times S(-\overline{\gamma}q_3''q',m\overline{\widehat{q}}^2\overline{q_3''}^2;q_3'),
\ena
and
\bna
\mathscr{C}_2^{**}(b_1,b_2,b_3,n,m,v;q_3'')&=&\sideset{}{^*}\sum_{\gamma \bmod q_3''}
e\left(\frac{-\overline{\gamma}v}{q_3''}\right)
G(\gamma,b_1;q_3'')G(\gamma,b_2;q_3'')G(\gamma,b_3;q_3'')\\
&&\times S(-\overline{\gamma}q_3'q',m\overline{\widehat{q}}^2\overline{q_3'}^2;q_3'').
\ena
We estimate $\mathscr{C}_2^{**}(b_1,b_2,b_3,n,m,v;q_3'')$ similarly as
$\mathscr{C}^{*}(b_1,b_2,b_3,n,m,v;q')$ getting
\bea
\mathscr{C}_2^{**}(b_1,b_2,b_3,n,m,v;q_3'')\ll q''^{3}\tau(q_3'').
\eea

To estimate $\mathscr{C}_1^{**}:=\mathscr{C}_1^{**}(b_1,b_2,b_3,n,m,v;q_3')$,
we factor $q_3'$ as $q_3'=p_1p_2\cdot\cdot\cdot p_s$, $p_i$ prime, and correspondingly,
\bea
\mathscr{C}_1^{**}=\prod_{i=1}^s\mathscr{T}(b_1,b_2,b_3,
q'q_3''p_i',m\overline{\widehat{q}}^2\overline{q_3''}^2\overline{p_i'}^2;p_i),
\eea
where $p_i'=q_3'/p_i$ and
\bna
\mathscr{T}(b_1,b_2,b_3,r_1,r_2m;p)
=\sideset{}{^*}\sum_{z\bmod p}e\left(\frac{-v\overline{z}}{p}\right)
G(z,b_1;p)G(z,b_2;p)G(z,b_3;p)S(-r_1\overline{z},r_2m;p)
\ena
with $(p,2r_1r_2)=1$.

By (9.3) and (9.4), we write
\bna
\mathscr{T}(b_1,b_2,b_3,r_1,r_2m;p)
=\epsilon_p^3p^{\frac{3}{2}}\sideset{}{^*}\sum_{z\bmod p}\left(\frac{z}{p}\right)
e\left(\frac{-\overline{4}(4v+b_1^2+b_2^2+b_3^2)\overline{z}}{p}\right)
S(-r_1\overline{z},r_2m;p).\\
\ena

By (9.5)-(9.9), Proposition 5.1 follows from the following lemma.

\begin{lemma}
We have
\bna
\mathscr{T}(b_1,b_2,b_3,r_1,r_2m;p)\ll p^{\frac{5}{2}}.
\ena
\end{lemma}

\begin{proof}
If $p|m$, then $S(-r_1\overline{z},r_2m;p)=-1$ and trivially,
\bna
\mathscr{T}(b_1,b_2,b_3,r_1,r_2m;p)=-\epsilon_p^3p^{\frac{3}{2}}\sideset{}{^*}
\sum_{z\bmod p}\left(\frac{z}{p}\right)
e\left(\frac{-\overline{4}(4v+b_1^2+b_2^2+b_3^2)\overline{z}}{p}\right)
\ll p^{\frac{5}{2}}.
\ena

If $p\nmid m$, we open the Kloosterman sum to obtain
\bna
\mathscr{T}(b_1,b_2,b_3,r_1,r_2m;p)
=\epsilon_p^3p^{\frac{3}{2}}\sum_{y,z\in \mathbb{F}_p^{\times}}\left(\frac{z}{p}\right)
e\left(\frac{-\overline{4}(4v+b_1^2+b_2^2+b_3^2)\overline{z}-r_1y\overline{z}
+r_2m\overline{y}}{p}\right).
\ena
If $p\nmid m$, $p|4v+b_1^2+b_2^2+b_3^2$, changing variable $y\overline{z}\rightarrow z$, we have
\bna
\mathscr{T}(b_1,b_2,b_3,r_1,r_2m;p)&=&\epsilon_p^3p^{\frac{3}{2}}\sum_{y\in \mathbb{F}_p^{\times}}\left(\frac{y}{p}\right)
e\left(\frac{r_2m\overline{y}}{p}\right)\sum_{z\in \mathbb{F}_p^{\times}}
\left(\frac{\overline{z}}{p}\right)
e\left(\frac{-r_1z}{p}\right)\nonumber\\
&=&\epsilon_p^3p^{\frac{3}{2}}\left(\frac{\overline{r_2m}}{p}\right)
\left(\frac{-\overline{r_1}}{p}\right)\tau\left(\left(\frac{\cdot}{p}\right)\right)^2,
\ena
where $\tau\left(\left(\frac{\cdot}{p}\right)\right)$ is the Gauss sum associated with
the quadratic residue $\left(\frac{\cdot}{p}\right)$, i.e.
\bna
\tau\left(\left(\frac{\cdot}{p}\right)\right)=\sum_{\gamma\bmod p}
\left(\frac{\gamma}{p}\right)e\left(\frac{\gamma}{p}\right)=\epsilon_p\sqrt{p}.
\ena
Therefore,
\bna
\mathscr{T}(b_1,b_2,b_3,r_1,r_2m;p)\ll p^{\frac{5}{2}}.
\ena

If $p\nmid m$, $p\nmid 4v+b_1^2+b_2^2+b_3^2$, we denote $r_0=-\overline{4}(4v+b_1^2+b_2^2+b_3^2)$
and let $f(y,z)=r_0z^{-1}-r_1z^{-1}y+r_2my^{-1}\in
\mathbb{F}_{p}^{\times}[y,z,(yz)^{-1}]$.
The Newton polyhedron $\Delta(f)$ of $f$ is the triangle in $\mathbb{R}^2$ with vertices
$(-1,0)$, $(-1,1)$ and $(0,-1)$. Thus $\mathrm{dim}\Delta(f)=2$. Moreover, for each of the following six
polynomials
\bna
f_{\sigma}(y,z)=r_0z^{-1}, -r_1z^{-1}y,
r_2my^{-1}, r_0z^{-1}-r_1z^{-1}y,r_0z^{-1}+r_2my^{-1},
-r_1z^{-1}y+r_2my^{-1}
\ena
corresponding to the faces of $\Delta(f)$ not containing $(0,0)$,
the locus of
\bna
\frac{\partial f_{\sigma}}{\partial y}=\frac{\partial f_{\sigma}}{\partial z}=0
\ena
is empty in $\left(\overline{F_p}^{\times}\right)^2$. In other words $f$ is non-degenerate
with respect to $\Delta(f)$. By Corollary 0.3 in Fu \cite{Fu}, we have
\bna
\sum_{y,z\in \mathbb{F}_p^{\times}}\left(\frac{z}{p}\right)
e\left(\frac{r_0\overline{z}-r_1\overline{z}y+r_2m\overline{y}}{p}\right)\ll p.
\ena
This completes the proof of Lemma 9.1.

\end{proof}

\medskip
\noindent
{\sc Acknowledgements.} The authors would like to thank Professor Lei Fu for valuable
discussion on the twisted character sums.
The first author is supported by
the National Natural Science Foundation of China (Grant No. 11101239) and the second author
is supported by the Natural Science Foundation of Shandong Province (Grant No. ZR2015AM010).

\bigskip

{\small \textsc{Qingfeng Sun},
\textsc{School of Mathematics and Statistics, Shandong University, Weihai,
Weihai, Shandong 264209, China}\\
\indent{\it E-mail address}: qfsun@sdu.edu.cn

\medskip

\textsc{Deyu Zhang},
\textsc{School of Mathematical Sciences, Shandong Normal University, Jinan,
Shandong 250014, China}\\
\indent{\it E-mail address}: zdy\_78@hotmail.com}

\end{document}